\documentclass{amsart}
\usepackage[utf8]{inputenc}
\usepackage{amsmath,amssymb,amsthm, graphics, comment,bm}
\usepackage{mathtools}
\usepackage{amsthm}
\usepackage{diagmac2}
\usepackage{tikz-cd}
\usepackage{textcomp}
\usepackage{color}
\usepackage[alphabetic]{amsrefs}
\usepackage[all,cmtip,color,matrix,arrow]{xy}
\usepackage{yfonts}
\usepackage{tikz}
\usepackage{stmaryrd}
\usepackage{mathrsfs}
\usepackage{enumerate}
\usepackage{hyperref}
\hypersetup{
	colorlinks=true,
	linkcolor=blue,
	filecolor=magenta,      
	urlcolor=cyan,
}
\usepackage[mathscr]{euscript}
\usepackage{cleveref}

\calclayout

\newcommand{\Proj}{\operatorname{Proj}}

\newcommand{\bA}{\mathbb{A}}

\newcommand{\bG}{\mathbb{G}}
\newcommand{\bL}{\mathbb{L}}

\newcommand{\bP}{\mathbb{P}}
\newcommand{\bQ}{\mathbb{Q}}

\newcommand{\bZ}{\mathbb{Z}}

\newcommand{\oH}{\operatorname{H}}

\newcommand{\sF}{\mathscr{F}}

\newcommand{\cB}{\mathcal{B}}

\newcommand{\cE}{\mathcal{E}}

\newcommand{\cG}{\mathcal{G}}

\newcommand{\cI}{\mathcal{I}}

\newcommand{\cL}{\mathcal{L}}
\newcommand{\cM}{\mathcal{M}}

\newcommand{\cO}{\mathcal{O}}
\newcommand{\cP}{\mathcal{P}}

\newcommand{\cS}{\mathcal{S}}
\newcommand{\cU}{\mathcal{U}}
\newcommand{\cW}{\mathcal{W}}
\newcommand{\cX}{\mathcal{X}}

\newcommand{\spec}{\operatorname{Spec}}

\newcommand{\Sym}{\operatorname{Sym}}

\newcommand{\GL}{\operatorname{GL}}

\newcommand{\PGL}{\operatorname{PGL}}
\newcommand{\SL}{\operatorname{SL}}
\newcommand{\im}{\operatorname{im}}

\newcommand{\bmu}{\bm{\mu}}
\newcommand{\ch}{\operatorname{CH}}
\newcommand{\pr}{{\rm pr}}
\newtheorem{theorem}{Theorem}[section]

\newtheorem*{theorem-no-num}{Theorem}
\newtheorem{Lemma}[theorem]{Lemma}

\newtheorem*{Oss'}{Observation}
\newtheorem{Cor}[theorem]{Corollary}
\newtheorem{Prop}[theorem]{Proposition}

\theoremstyle{definition}

\newtheorem{Def}[theorem]{Definition}

\newtheorem{Remark}[theorem]{Remark}

\title{The integral Chow rings of moduli of Weierstrass fibrations}
\author[S. Canning]{Samir Canning}
    \address[S. Canning]{University of California San Diego}
    \email{srcannin@ucsd.edu}
\author[A. Di Lorenzo]{Andrea Di Lorenzo}
	\address[A. Di Lorenzo]{Humboldt Universit\"{a}t zu Berlin, Germany}
	\email{andrea.dilorenzo@hu-berlin.de}
\author[G. Inchiostro]{Giovanni Inchiostro}
    \address[G. Inchiostro]{University of Washington}
    \email{ginchios@uw.edu}
\begin{document}
\begin{abstract}
 We compute the Chow rings with integral coefficients of moduli stacks of minimal Weierstrass fibrations over the projective line. For each integer $N\geq 1$, there is a moduli stack $\cW^{\mathrm{min}}_N$ parametrizing minimal Weierstrass fibrations with fundamental invariant $N$. Following work of Miranda and Park--Schmitt, we give a quotient stack presentation for each $\cW^{\mathrm{min}}_N$. Using these presentations and equivariant intersection theory, we determine a complete set of generators and relations for each of the Chow rings. For the cases $N=1$ (respectively, $N=2$), parametrizing rational (respectively, K3) elliptic surfaces, we give a more explicit computation of the relations.  
\end{abstract}

\maketitle

\section{Introduction}
The study of the Chow rings of moduli spaces has played a central role in algebraic geometry ever since Mumford's introduction of an intersection product for the moduli space of curves $\cM_g$ and its compactification by stable curves. Mumford's intersection product requires the use of \emph{rational coefficients}, but Totaro \cite{Tot} and Edidin--Graham \cite{EG} developed an intersection theory for quotient stacks that works with \emph{integral coefficients}. Many moduli stacks of interest in algebraic geometry, including the moduli stacks of curves, are quotient stacks. 

Chow rings with integral coefficients are often quite difficult to compute, but in turn they have a much richer structure than their rational counterparts: for instance, rational Chow rings of moduli of hyperelliptic curves are trivial, but the integral ones are not (see \cites{Vis, EdFu, DL}). 

Only a few examples have been computed in full for moduli stacks of curves $\cM_{g,n}$ and $\overline{\cM}_{g,n}$ with $g$ and $n$ small (see \cites{DLFV, Lar, DLV, DLPV, Inc}), and even less is known for moduli stacks parametrizing higher dimensional varieties.

In this paper, we study \emph{integral} Chow rings of certain moduli stacks of surfaces that we denote $\cW^{\min}_N$, indexed by an integer $N\geq 1$. The stacks $\cW^{\min}_N$ parametrize surfaces called \emph{minimal Weierstrass fibrations} over $\bP^1$. The arithmetic and geometry of moduli spaces of minimal Weierstrass fibrations over $\bP^1$ has already been the subject of investigation of several works (see for instance \cites{Mir, HP, PS, CK}). Moreover, the stack $\cW^{\min}_2$ is of particular interest, as it can be regarded as the moduli stack of elliptic K3 surfaces with a section (equivalently, K3 surfaces polarized by a hyperbolic lattice).

Minimal Weierstrass fibrations over $\bP^1$ are flat, proper morphisms $p:X\rightarrow \bP^1$ together with a section $s:\bP^1\rightarrow X$ satisfying the following conditions:
\begin{enumerate}
    \item $X$ is normal, irreducible, with at most ADE singularities;
    \item every fiber of $p$ is isomorphic to an elliptic curve, a rational curve with a node, or a rational curve with a cusp;
    \item the section does not intersect the singular point of any of the fibers.
\end{enumerate}
These fibrations arise naturally from contracting the components of the fibers of a smooth elliptic surface over $\bP^1$ that do not meet the section.
Associated to a minimal Weierstrass fibration is a fundamental invariant $N=\deg (R^1p_*\cO_X)^{\vee}\geq 0$. For each $N\geq 1$, we consider moduli stacks $\cW^{\min}_N$ parametrizing minimal Weierstrass fibrations with fundamental invariant $N$. 

Our main result is the following. For a more precise formulation, see \Cref{thm:main}.
\begin{theorem}\label{main}
Suppose that the ground field has characteristic $\neq 2,3$ and let $N\geq 1$ be an integer. Then
\begin{enumerate}
    \item for $N$ odd, we have
    \[ \ch^*(\cW^{\min}_N)\simeq \bZ[c_1,c_2]/I_N \]
    where the generators are Chern classes of a certain rank two vector bundle $\cE_N$ and the ideal of relations $I_N$ is generated by $\binom{N+2}{2}$ relations, of which one has degree $8N+1$ and the others have degree $9k+m$ for $1\leq k\leq N$ and $0\leq m\leq k$. Explicit formulas for these relations are given in (\ref{eq:relations D2 for GL2}).\\
    \item for $N$ even, we have
    \[ \ch^*(\cW^{\min}_N)\simeq \bZ[\tau_1,c_2,c_3]/(2c_3,I_N) \]
    where the generators are Chern classes of certain vector bundles $\cL_N$ and $\cE_N$ and the ideal of relations $I_N$ is generated by $\binom{N+2}{2}$ relations, of which one has degree $8N+1$ and the others have degree $9k+m$ for $1\leq k\leq N$ and $0\leq m\leq k$. Explicit formulas for these relations are given in (\ref{eq:relations D2 for PGL2 k even}) and (\ref{eq:relations D2 for PGL2 k odd}).
\end{enumerate}
\end{theorem}

Perhaps the most interesting cases are when $N=1$ and $N=2$. Minimal Weierstrass fibrations with fundamental invariant $1$ are rational. They arise from the elliptic fibrations obtained by blowing up the base points of a pencil of cubics in $\bP^2$. Moduli spaces of rational elliptic fibrations are well studied, and in particular are closely related to several other interesting moduli problems \cite{Vakil}.

Minimal Weierstrass fibrations with fundamental invariant $2$ come from elliptic K3 surfaces. The intersection theory with rational coefficients of moduli spaces of K3 surfaces has been the subject of much recent research and is expected to behave analogously to that of the moduli space of curves \cite{CK, MOP, PY}. 

Specializing our Theorem \ref{main} to $N=1,2$, we obtain the following completely explicit result.
\begin{theorem}\label{explicit}
Suppose that the ground field has characteristic $\neq 2,3$. Then
\begin{enumerate}
    \item the integral Chow ring of the moduli stack $\cW^{\min}_1$ of rational elliptic surfaces with a section is
    \[ \bZ[c_1,c_2]/(6c_1c_2r_6,c_1^3r_6,c_1^2c_2r_6) \]
    where the generators are Chern classes of a certain rank two vector bundle $\cE_1$ and $$r_6=576(30c_1^6+151c_1^4c_2+196c_1^2c_2^2+64c_2^3);$$
    \item the integral Chow ring of the moduli stack $\cW^{\min}_2$ of elliptic K3 surfaces with a section is
    \[ \bZ[\tau_1,c_2,c_3]/(2c_3,r_{9},r_{10},r_{18},r_{19}) \]
    where the generators are Chern classes of certain vector bundles $\cL_2$ and $\cE_2$ and
   \begin{align*}
    r_9=&1152 (691 c_2^4 \tau_1 - 38005 c_2^3 \tau_1^3 + 309568 c_2^2 \tau_1^5 - 
   497520 c_2 \tau_1^7 + 124416 \tau_1^9), \\
   r_{10}=&1152 (30 c_2^5 - 6811 c_2^4 \tau_1^2 + 133495 c_2^3 \tau_1^4 - 
   481528 c_2^2 \tau_1^6 + 327600 c_2 \tau_1^8 - 20736 \tau_1^{10}), \\
   r_{18}=&1152 c_2^5 (108314154642930 c_2^4 + 1045672 c_2^3 \tau_1^2 - 
   89483 c_2^2 \tau_1^4 + 35 c_2 \tau_1^6 - 4 \tau_1^8),\\
   r_{19}=&2304 c_2^6 \tau_1 (118203201 c_2^3 + 180502 c_2^2 \tau_1^2 - 7 c_2 \tau_1^4 + 4 \tau_1^6).
\end{align*}
\end{enumerate}
\end{theorem}

When $N\geq 2$, the rational Chow rings $\ch^*(\cW^{\min}_N)\otimes \bQ$ have been computed by the first author and Kong \cite{CK}: in particular, they proved that only $r_{9}$ and $r_{10}$ are needed in order to generate the ideal of relations \emph{with rational coefficients}. This implies that there are many torsion classes in $\ch^*(\cW^{\min}_2)$, so the theory with integral coefficients is genuinely different from that with rational coefficients and contains much more information.

\subsection*{Structure of the paper}
In Section \ref{moduli}, we construct the moduli stacks $\cW^{\min}_N$ as quotient stacks, following the work of Miranda \cite{Mir} who constructed coarse spaces for $\cW^{\min}_N$ and Park--Schmitt \cite{PS}, who constructed $\cW^{\min}_N$ as quotients of a weighted projective stack by $\PGL_2$. Our approach is slightly different from that of Park--Schmitt, but we will show that the two approaches coincide. 

In Section \ref{equivint}, we discuss the equivariant intersection theory of projective spaces, which is a key tool in the proof of Theorem \ref{main}. In particular, we obtain generators for $\ch^*(\cW^{\min}_N)$. 

In Sections \ref{Delta1} and \ref{Delta2}, we compute relations among the generators of the Chow ring that result from excising the locus of non-minimal Weierstrass fibrations, finishing the proof of Theorem \ref{main}. 

Finally, in Section \ref{N12} we make explicit calculations of the relations in the cases $N=1$ and $N=2$, proving Theorem \ref{explicit}.

\subsection*{Acknowledgments}
We are thankful to Jarod Alper, Elham Izadi, Bochao Kong, Hannah Larson, Johannes Schmitt, and Angelo Vistoli for helpful conversations. We especially thank the referee for the many comments and corrections improving the exposition. Part of this material is based upon work supported by the Swedish Research Council under grant no. 2016-06596 while the first and the second named authors were in residence at Institut Mittag-Leffler in Djursholm, Sweden during the fall of 2021. S.C. was partially supported by NSF--RTG grant DMS-1502651. 

\section{Moduli of minimal Weierstrass fibrations}\label{moduli}
In this Section, after recalling some basic definitions, we give in \Cref{thm:presentation WN full} a presentation as a quotient stack of $\cW^{\min}_N$, the moduli stack of minimal Weierstrass fibrations with fundamental invariant $N$. We also introduce some vector bundles $\cL_N$ and $\cE_N$ (see \Cref{def:vec N odd} and \Cref{def:vec N even}) which will be relevant for our computations.

We adopt the following convention: given a locally free sheaf $\cE$ over a scheme $X$, the associated vector bundle $E$ is given by $\underline{\spec}_{\cO_X}(\underline{\Sym}(\cE^\vee))$ and the projectivization of this vector bundle $\bP(E)$ is therefore $\underline{\Proj}_{\cO_X}(\underline{\Sym}(\cE^\vee))$. Observe that with this convention, given $\pi:\bP(E)\to X$, we have $\pi_*\cO_{\bP(E)}(1)\simeq \cE^\vee$. 

\subsection{Basic notation}\label{subsec:basic}
We collect here some basic notation that is used throughout the paper; in this section, we work over $\spec(\bZ)$.
Let $G_N:=\operatorname{Aut}(\mathbb{P}^1,\mathcal{O}(N))$ be the automorphism group of the pair $(\mathbb{P}^1,\mathcal{O}(N))$. More precisely, an automorphism $\phi=(\phi_0,\phi_1)$ consists of a pair where $\phi_0:\mathbb{P}^1\overset{\sim}{\longrightarrow}\mathbb{P}^1$ is an automorphism of $\mathbb{P}^1$ and $\phi_1:\mathcal{O}(N)\overset{\sim}{\longrightarrow}\phi_0^*\mathcal{O}(N)$ is an automorphism of line bundles.

We make the group $G_N$ act linearly on the $\bZ$-module $V^N_N:=\oH^0(\mathbb{P}^1,\mathcal{O}(N))$ as follows: given a global section $\sigma:\mathcal{O}\to\mathcal{O}(N)$, we define $\phi\cdot\sigma=(\phi_0,\phi_1)\cdot\sigma$ to be the composition
\[\cO \overset{{\rm can}}{\longrightarrow} \phi_0^*\cO \overset{\phi_0^*\sigma}{\longrightarrow} \phi_0^*\cO(N) \overset{\phi_1^{-1}}{\longrightarrow} \cO(N).  \]
Given any positive integer $r$, we have a surjective homomorphism
\[ p_{Nr}^{N}:G_N \longrightarrow G_{Nr},\quad \phi=(\phi_0,\phi_1) \longmapsto (\phi_0,\phi_1^{\otimes r}).\]
We can use this homomorphism to define an action of $G_N$ on $V^{Nr}_{Nr}$ as
\[ \sigma  \longmapsto (\phi_0,\phi_1^{\otimes r}) \cdot \sigma. \]
In what follows, we will denote this $G_N$-module as $V_{Nr}^{N}$. We will also use the notation $V^{N}_{Nr_1,Nr_2}$ to denote $V^N_{Nr_1}\oplus V^N_{Nr_2}$.

The $G_N$-modules that we just introduced can be made more explicit as follows. First, observe that $G_1\simeq \GL_2$, and that there is a surjective morphism
\[G_1 \simeq \GL_2 \longrightarrow G_N,\quad (\phi_0,\phi_1) \longmapsto (\phi_0,\phi_1^{\otimes N}) \]
whose kernel consists of the subgroup of roots of unity $\bmu_N\subset \GL_2$, embedded diagonally. The action of $G_N$ on $V^N_N$, after identifying $G_N$ with $\GL_2/\bmu_N$, can then be written as 
\[ [A] \cdot f((x,y)):=f(A^{-1}(x,y)).\]
Moreover, by \cite[Proposition 4.4]{AV04} we have isomorphisms 
\begin{align*} 
&\GL_2 \overset{\simeq}{\longrightarrow} \GL_2/\bmu_N,\quad A \longmapsto  [\det(A)^{-\frac{N-1}{2N}} A],& \text{for }N\text{ odd}, \\
&\bG_m\times\PGL_2 \overset{\simeq}{\longrightarrow} \GL_2/\bmu_N, \quad (\alpha,A) \longmapsto [\alpha^\frac{1}{N}\det(A)^{-\frac{1}{2}}A],& \text{for }N\text{ even}.
\end{align*}
We can use these isomorphisms to describe the $G_N$-modules $V^N_{Nr}$ as $\GL_2$-modules (resp. $\bG_m\times\PGL_2$) for $N$ odd (resp. for $N$ even). Denoting by $E$ the standard $\GL_2$-module, and by $L$ the standard $\bG_m$-module with trivial $\PGL_2$-action, we obtain:
\begin{align*}
    &V_{Nr}^{N} \simeq \det(E)^{r\frac{N-1}{2}} \otimes \Sym^{N}(E^{\vee}), &  \text{for }N\text{ odd}, \\
    &V_{Nr}^{N} \simeq  L^{-r}\otimes V_{Nr}^{Nr} \simeq L^{-r}\otimes \oH^0(\bP^1,\cO(Nr)), &  \text{for }N\text{ even},
\end{align*}
where in the last line we are endowing $\oH^0(\bP^1,\cO(Nr))$ with the trivial $\bG_m$-action and the obvious $\PGL_2$-action.

%We denote $V_m:=\oHs^0(\bP^1,\cO_{\bP^1}(m))$. Moreover, we denote
%\[ V_{m_1,m_2}:=\oH^0(\bP^1,\cO_{\bP^1}(m_1))\oplus \oH^0(\bP^1,\cO_{\bP^1}(m_2)).\]

%The symbol $G$ will be used to denote the group $\GL_2/\bmu_N$ for some $N$. Observe that for $N$ odd we have $G\simeq\GL_2$ and for $N$ even we have $G\simeq\PGL_2\times\bG_m$ (see \cite{AV04}*{Proposition 4.4}). We will consider the following $G$-actions
%\begin{itemize}
    %\item when $G=\GL_2$, we identify $V_{m}$ with $\Sym^mE^{\vee}$, where $E$ is the standard $\GL_2$-representation. In other terms, we let $\GL_2$ acts on $V_m$ via the formula $A\cdot f(x,y):=f(A^{-1}(x,y))$;
    %\item when $G=\PGL_2\times\bG_m$, we regard $V_{2m}$ as the representation on which $\PGL_2$ acts via the formula $A\cdot f(x,y):=\det(A)^mf(A^{-1}(x,y))$ and $\bG_m$ acts trivially. This definition does not depend on the choice of a representative $A$ for the element in $\PGL_2$.
%\end{itemize}
%The symbol $V_{2dm}'$ will stand for the following $G$-representations:
%\begin{itemize}
    %\item when $G=\GL_2$ we set $V'_{2dm}:=V_{2dm}\otimes\det(E)^{\otimes d(m-1)}$.
    %\item when $G=\PGL_2\times\bG_m$ we set $V'_{2dm}:=V_{2dm}\otimes L^{\otimes(-2d)}$, where $L$ is the standard representation of $\bG_m$.
%\end{itemize}

\subsection{Stacks of conics with sections}
In this section, we keep working over $\spec(\bZ)$. We start by introducing a stack, denoted $\sF^{\min}_N$, which admits a natural presentation as a quotient stack (\Cref{thm:presentation WN}). The reason for introducing this stack is that, once restricted to $\spec(\bZ[\frac{1}{6}])$, it will turn out to be isomorphic to $\cW^{\min}_N$, the moduli stack of minimal Weierstrass fibrations.
\begin{Def}\label{def Fmin}
We denote by $\sF^{\min}_N$ the following fibered category over $\mathfrak{S}ch/\bZ$, the category of schemes over $\spec(\bZ)$.

\begin{bf} Objects:\end{bf} The objects over a scheme $T$ are tuples $(f:\cP\to T, \cL, A, B)$ consisting of a flat, proper morphism of finite presentation $\cP\to T$ with geometric fibers isomorphic to $\bP^1$, a line bundle $\cL$ over $\cP$ of degree $N$ along each fiber of $f$, and two sections $A,B$ of $\oH^0(\cP,\cL^{\otimes 4})$ and $\oH^0(\cP,\cL^{\otimes 6})$ respectively. We require the sections $A$  and $B$ to satisfy the following two conditions:
\begin{enumerate}
    \item for each geometric point $s$ of $T$, the global section $4A^3+27B^2$ of $\cL^{\otimes 12}$ is not zero once restricted to $\cP_s$, and
    \item for each geometric point $s$ of $T$, there is no point $p$ of $\bP^1\simeq\cP_s$ such that $A_s$ (resp. $B_s$) vanishes in $p$ with order $\geq 4$ (resp. with order $\geq 6$).
\end{enumerate}

\begin{bf} Morphisms:\end{bf} A morphism $(f:\cP\to T, \cL, A, B)\to (f':\cP'\to T', \cL', A', B')$ consists of a morphism $T\to T'$, together with two isomorphisms $\phi:\cP\to \cP'\times_{T'}T$ and $\psi:\cL\to\phi^*\cL'$, such that $\psi$ sends $A$ (resp. $B$) to $A'$ (resp. $B'$).
\end{Def}
\begin{Def}\label{nonmimimal}
We define the $G_N$-invariant closed subscheme $\Delta_N$ in $V^N_{4N,6N}$ as the union of $\Delta^1_N$ and $\Delta^2_N$, where
\begin{itemize}
    \item the subscheme $\Delta_N^1$ is the locus of pairs $(A,B)$ such that $4A^3 + 27B^2 = 0$, and
    \item the subscheme $\Delta^2_N$ is the locus of pairs $(A,B)$ such that there exists a point $p\in\bP^1$ such that $A$ (resp. $B$) vanishes in $p$ with order $\geq 4$ (resp. with order $\geq 6$).
\end{itemize}
\end{Def}
The following Proposition gives a presentation of $\sF^{\min}_N$ as a quotient stack.
\begin{Prop}\label{thm:presentation WN}
There is an isomorphism $\sF^{\min}_N \cong [(V^N_{4N,6N}\smallsetminus \Delta_N)/G_N]$.
\end{Prop}

\begin{proof}
Our argument follows \cite[Theorem 4.1]{AV04}. It suffices to construct a map $(V^N_{4N,6N}\smallsetminus\Delta_N)\to \sF^{\min}_N$ which is a $G_N$-torsor.

The data of a map $T\to V^N_{4N,6N}$ is equivalent to a section of the projection $\pi_2:V^N_{4N,6N}\times T\to T$. Let  $p:\bP^1\times T\to T$ be the second projection. Since $\pi_2$ is affine, a section of $\pi_2$ induces a morphism
\[ \operatorname{Sym}_{\cO_T}^{\bullet}(p_*(\cO_{\bP^1\times T}(4N)\oplus \cO_{\bP^1\times T}(6N))^\vee) \to  \cO_T.\]
This is the same as a map $p_*(\cO_{\bP^1\times T}(4N)\oplus \cO_{\bP^1\times T}(6N))^\vee \to \cO_T$ that in turn is equivalent to a map $\cO_T\to p_*(\cO_{\bP^1\times T}(4N)\oplus \cO_{\bP^1\times T}(6N))$, namely a choice of a pair of sections $A\in \oH^0(T,p_*\cO_{\bP^1\times T}(4N))$ and $B\in \oH^0(T,p_*\cO_{\bP^1\times T}(6N))$, equivalently a pair of sections of $\oH^0(\bP^1\times T,\cO_{\bP^1\times T}(4N))$ and $ \oH^0(\bP^1\times T,\cO_{\bP^1\times T}(6N))$.

In particular, the data of a morphism $T\to V^N_{4N,6N}\smallsetminus\Delta_N$ is equivalent to the data of two sections $(A,B)$ of $\oH^0(\bP^1\times T,\cO_{\bP^1\times T}(4N))$ and $ \oH^0(\bP^1\times T,\cO_{\bP^1\times T}(6N))$ such that for each geometric point $s$ of $T$, the restrictions of $A$ and $B$ over $\bP^1_s$ do not verify the two conditions given in \Cref{nonmimimal}. 

There is a natural transformation $\Phi: (V^N_{4N,6N}\smallsetminus\Delta_N)\to \sF^{\min}_N$, that on objects is defined as follows. Given a map $T\to V^N_{4N,6N}\smallsetminus\Delta_N$, which corresponds to two sections $A,B$ as above, we can associate the object of $\sF^{\min}_N$ given by $(\bP^1\times T\to T, \cO_{\bP^1\times T}(N), A,B)$.

Let $\sigma:G_N\times (V^N_{4N,6N}\smallsetminus\Delta_N)\to (V^N_{4N,6N}\smallsetminus\Delta_N)$ be the map that defines the action of $G_N$ on $V^N_{4N,6N}\smallsetminus\Delta_N$, and denote $\pr_2:G_N\times (V^N_{4N,6N}\smallsetminus\Delta_N)\to (V^N_{4N,6N}\smallsetminus\Delta_N)$ the projection on the second factor.

We claim that $\Phi$ is a $G_N$-torsor. We need to show that:
\begin{enumerate}
    \item The two arrows $\Phi\circ \sigma$ and $ \Phi \circ \pr_2$ are isomorphic,
    \item For every scheme $T$ and every object $(f:\cP\to T, \cL, A, B)$ of $\sF^{\min}_N(T)$, there is an \'etale cover $T'\to T$ such that the pull-back $(f':\cP'\to T', \cL', A', B')$ is isomorphic to an object of $(V^N_{4N,6N}\smallsetminus\Delta_N)(T)$ (i.e. it is in the essential image of $\Phi$), and
    \item If $\alpha:= (f':\cP'\to T', \cL', A', B')$ is in the essential image of $\Phi$, the action of $G_N$ on its essential fiber (i.e. the pairs $(\beta, \gamma)$ consisting of an element $ \beta \in (V^N_{4N,6N}\smallsetminus\Delta_N)(T)$ and an isomorphism $\Phi(\beta)\to \alpha$) is simply transitive.
\end{enumerate}
To check point (1), we construct explicitly the isomorphism: for every $(\phi,A,B)$ in $G_N(T)\times (V_{4N,6N}^n\smallsetminus\Delta_N)$ we define the 2-morphism
\begin{align*}
    (\Phi\circ\pr_2)((\phi, A, B)) = (\bP^1_T\to T,\cO(N), A, B) \longmapsto (\bP^1_T \to T, \cO(N),\phi\cdot A, \phi\cdot B)= (\Phi\circ\sigma)((\phi, A, B)).
\end{align*} 
to be exactly $(\phi_0,\phi_1)$, where $\phi\cdot A$ (resp. $\phi\cdot B$) is the action introduced in \Cref{subsec:basic}.
To check (2), observe that $f:\cP\to T$ is a Severi-Brauer scheme \cite[Corollaire 8.3]{GroBr}, hence there is an \'etale cover $T'\to T$ and an isomorphism $\cP\times_T T' \cong \bP^1\times T'$ over $T'$. Then if we denote by $\cL'$ the pull-back of $\cL$ to $\bP^1\times T'$ we have two line bundles, $\cL'$ and $\cO_{\bP^1\times T'}(N)$ that are isomorphic along each fiber: in particular, for every point $s$ of $T$ we have $\oH^1(\bP^1\times\{s\},\cL'(-N)_s)=0$, hence \cite[Theorem III.12.11]{Har13} we obtain that $\cG:=\pr_{2*}(\cL'\otimes \cO_{\bP^1\times T'}(-N))$ is a line bundle, from which we immediately deduce that, up to replacing $T'$ with a covering that trivializes $\cG$, we can assume that $\cL'\cong \cO_{\bP^1\times T'}(N)$. This proves point (2).

To check point (3) it suffices to recall that the functor sending a scheme $T$ to $\underline{\operatorname{Aut}}_T(\bP^1_T,\cO_{\bP^1_T}(N))$ is represented by $G_N$ (see \cite[Proof of Theorem 4.1]{AV04}). Indeed, we need to check that:
\begin{itemize}
    \item The action of $G_N$ is transitive on the fibers of $(V^N_{4N,6N}\smallsetminus\Delta_N) \to \sF^{\min}_N$, and
    \item The action is simply transitive (this is analogous to the representability of $[(V^N_{4N,6N}\smallsetminus\Delta_N)/G_N]\to \sF^{\min}_N$).
\end{itemize}
To check the first bullet point, we need to check that if two objects of $\sF^{\min}_N(T)$ that belong to the image of $\Phi(T)$ are isomorphic, then there is an element of $G_N(T)$ which sends the first one to the second one. To check the second bullet point, we need to check that such an element is unique.
Both bullet points follow since $\underline{\operatorname{Aut}}_T(\bP^1_T,\cO_{\bP^1_T}(N))$ is represented by $G_N$. 
\end{proof}

\subsection{Moduli of Weierstrass fibrations}
In this section, we work over $\spec(\bZ[\frac{1}{6}])$. Let $\cW_N^{\min}$ be the moduli stack of minimal Weierstrass fibrations, as defined in \cite{PS}*{Section 4.2}. We will prove in \Cref{prop:WN iso FN} that $\cW^{\min}_N$ is isomorphic to the stack $\sF^{\min}_N$ that we introduced before. We start by recalling the relevant definitions from \emph{loc. cit.}.
\begin{Def}
    
\end{Def}
\begin{itemize}
    \item A \emph{Weierstrass fibration} over an algebraically closed field $k$ is a proper, flat morphism $f:X\to \bP^1_k$ with geometrically integral fibers from an integral scheme $X$ together with a section $s:\bP^1_k \to X$ such that every geometric fiber is either an elliptic curve, a rational curve with a node or a rational curve with a cusp, the generic fiber is smooth and the section $s(\bP^1_k)$ does not contain any singular point of the fibers.
    \item A Weierstrass fibration has \emph{degree} $N$ if the line bundle $(R^1f_*\cO_X)^{\vee}$ has degree $N$.
    \item A Weierstrass fibration is \emph{minimal} if it is a Weierstrass model of a smooth elliptic surface over $\bP^1_k$ with a section (see \cite[Section 1]{Miranda} for more details).
\end{itemize}
\begin{Def}
A \emph{family of minimal Weierstrass fibrations} of degree $N$ over a scheme $T$ is the data of:
\begin{enumerate}
    \item a flat, proper morphism of finite presentation $\cP\to T$ with geometric fibers isomorphic to $\bP^1$, and
    \item a flat, proper morphism of finite presentation $f:\cX\to \cP$ with a section $\cS\subseteq \cX$.
\end{enumerate}
We require that for every geometric point $p\in T$, the fiber $(\cX_p,\cS_p)\to \cP_p$ is a minimal Weierstrass fibration of degree $N$, and we refer the reader to \cite{Miranda} for a more detailed exposition on Weierstrass fibrations.
\end{Def}
Given two families $((\cX,\cS)\to \cP\to T)$ and $((\cX',\cS')\to \cP'\to T')$, a morphism from the latter to the former consists of a morphism $g:T'\to T$ and isomorphisms $\cX'\cong \cX\times_T T'$ and $\cP'\cong \cP\times_T T'$ which preserve the section and make the obvious square commutative.

It is shown in \cite[Theorem 1.2]{PS} that there is an algebraic stack, which we denote by $\cW^{\min}_N$, that parametrizes families of minimal Weierstrass fibrations. Our goal is to prove that $\cW^{\min}_N\cong \sF^{\min}_N$. We need the following preparatory Lemma, which is proved in \cite{Miranda}*{pages 22-24}.
\begin{Lemma}\label{lemma_miranda}
Let $((\cX,\cS)\xrightarrow{f} \cP\to T)$ be a family of minimal Weierstrass fibrations over $T$. Then:
\begin{itemize}
    \item $R^1f_*\cO_{\cX}$ is a line bundle, the dual of which will be denoted by $\cL$, 
    \item the inclusion $\cO_{\cX}(-S)\subset \cO_{\cX}$ induces an isomorphism $f_*\cO_{\cX}(\cS)\cong f_*\cO_{\cX}$, and
    \item for every $n\geq 2$ we have an exact sequence \[0\to f_*\cO_{\cX}((n-1)\cS)\to f_*\cO_{\cX}(n\cS) \to f_*\cO_{\cS}(n\cS) \to 0.\]
    Moreover, the sequence above splits and we have $f_*\cO_{\cX}(n\cS) = \cO_\cP\oplus \cL^{\otimes -2}\oplus ... \oplus \cL^{\otimes -n}$.
\end{itemize}
\end{Lemma}
The following is just a relative version of the arguments in \cite[II.5]{Miranda}. We report them below for convenience of the reader. 

Consider $((\cX,\cS)\xrightarrow{\phi} \cP\to T)$ be a family of minimal Weierstrass fibrations over a scheme $T$. First, we choose a covering of $\cP$ which trivializes $\cL$, and we choose a generator $e_1$ for $\cL^{-1}$. In particular, $e_n:=e_1^{\otimes n}$ will be a generator for $\cL^{-n}$. From Lemma \ref{lemma_miranda} this covering also trivializes $\phi_*\cO_{\cX}(n\cS)$, and we choose an element $f$ of $\phi_*\cO_{\cX}(2\cS)$ (resp. $g$ of $\phi_*\cO_{\cX}(3\cS)$)
that via the projection to $\phi_*\cO_{\cS}(2\cS)$ (resp. $\phi_*\cO_{\cS}(3\cS)$) of Lemma \ref{lemma_miranda} maps to $e_2$ (resp. $e_3$). Then $g^2$ and $f^3$ are sections of $\phi_*\cO_{\cX}(6\cS)$, and $g^2 = f^3  + h$ where $h$ maps to 0 via the projection $\phi_*\cO_{\cX}(6\cS)\to \phi_*\cO_{\cS}(6\cS)$. 

Proceeding as in \cite[II.5]{Miranda} (i.e. completing the square and the cube), locally in $\cP$ there exists unique regular functions $a$ and $b$ such that we can (still locally) choose $f$ and $g$ with $g^2 = f^3 + af +b$. If we pick another trivialization $e'_1=\lambda e_1$ for $\cL^{-1}$, the regular functions $a$ and $b$ change into $a'=\lambda^{4} a$ and $b'=\lambda^{6} b$: in particular, we have that $a'\cdot e'_{-4}= a\cdot e_{-4}$ (resp. $b' \cdot e_{-6}' = b\cdot e_{-6}$), hence we obtain a well defined global section $A$ of $\cL^{\otimes 4}$ (resp. a global section $B$ of $\cL^{\otimes 6}$).

For every point $x$ in $T$, the smoothness of the generic fiber of $\cX_x\to\cP_x$ is equivalent to imposing that the global section $4A_x^3 + 27B_x^2$ is not zero.
%Using \cite[II.5.3]{Miranda} we see that the choice of $A$ and $B$ is not unique (for example, it depends on the choice of $e_1$). However, given two different choices $(A,B)$ and $(A',B')$ there is an invertible function $\lambda \in \oH^0(\cP,\cO_{\cP}^*)$ such that $A = \lambda^{-4}A'$ and $b = \lambda^{-6} B'$.
Moreover, from \cite[Corollary 2.5]{Mir}, there is no point $p$ in a fiber of $\cP\to T$ where the order of vanishing of $A$, at $p$ is greater than 4 and the order of vanishing of $B$ is greater than 6 (as $\phi$ is a family of \emph{minimal} Weierstrass fibrations). Note that here, for $A$ and $B$ we intend the restriction of the sections to the fiber of $\cP\to T$ containing $p$.

Combining the previous paragraph with Lemma \ref{lemma_miranda}, we have
\begin{Cor}\label{Cor_ma_from_PS_stack_to_W_N}
Consider $((\cX,\cS)\xrightarrow{\phi} \cP\to T)$ a family of minimal Weierstrass fibrations over $T$. Then:
\begin{enumerate}
    \item The sheaf $\cL = (R^1\phi_*\cO_{\cX})^{\vee}$ is a line bundle,
    \item from the data above we can canonically construct two sections $A,B$ of $\oH^0(\cP,\cL^{\otimes 4})$ and $\oH^0(\cP,\cL^{\otimes 6})$,
    %\item Given two different choices $(A,B)$ and $(A',B')$ there is an invertible function $\lambda \in \oH^0(\cP,\cO_{\cP}^*)$ such that $A = \lambda^{-4}A'$ and $B = \lambda^{-6} B'$, and
    \item for every $x\in T$, the section $4A_x^3+27B_x^2$ is not zero on $\cP_x$, and
    \item for every $x\in T$, there is no point $y\in \cP_x$ such that the sections $A_x$ and $B_x$ of $\oH^0(\cP_x,\cL^{\otimes 4}_{|\cP_x})$ and $\oH^0(\cP_x,\cL^{\otimes 6}_{|\cP_x})$ vanish at $y$ with order $\ge 4$ and $\ge 6$, respectively.
\end{enumerate}
\end{Cor}

Corollary \ref{Cor_ma_from_PS_stack_to_W_N} gives a map $\cW^{\min}_N\to \sF^{\min}_N$.
We show that this map is an isomorphism, by producing an inverse. 

Given a family $(f:\cP\to T, \cL, A,B)$, let $E$ be the vector bundle associated to the the locally free sheaf $\cL^{-2}\oplus \cL^{-3}\oplus \cO_\cP$: then we can construct a family of Weierstrass fibrations by taking a closed subscheme of $\bP(E)$ as follows.

First consider a covering $\cU\to \cP$ which trivializes $\cL$, and let $s$ be the trivializing section of $\cL$. We can therefore write the pullback of $A$ (respectively $B$) as $a\cdot s^4$ (respectively $b\cdot s^6$). Then consider the closed subscheme of $\bP(E|_{\cU})$ given by those lines generated by $(x\cdot s^2,y\cdot s^3,z)$ with $x,y,z\in \cO_{\cU}$ such that $(y\cdot s^3)^2 z=(x\cdot s^2)^3+(a\cdot s^4)(x\cdot s^2)z^2+(b\cdot s^6)z^3$ (recall that with our convention we have $\bP(E)=\underline{\Proj}_{\cO_\cU}(\underline{\Sym}(\cL^{2}|_{\cU}\oplus\cL^{3}|_{\cU}\oplus\cO_{\cU}))$).

One can check that these closed subschemes descend to a closed subscheme $\cX\subseteq \bP(E)$. The map $\cX\to \cP$ has a section $\cS\subseteq \cX$, that over $\cU$ is given by $z = x\cdot s^2 = 0$. To check that this  is a family in $\cW^{\min}_N$ we need to check that when $T=\spec(k)$ for an algebraically closed field $k$, the resulting surface $\cX$ with section $\cS$ is a minimal Weierstrass fibration. The fact that it is a Weierstrass fibration follows from \cite[pg. 26]{Miranda}, whereas minimality follows from \cite[Corollary 2.5]{Mir}. We have proven the following.
\begin{Prop}\label{prop:WN iso FN}
We have an isomorphism $\cW_N^{\min}\simeq\sF^{\min}_N$.
\end{Prop}
Combining the result above with \Cref{thm:presentation WN}, we obtain the main result of this Section.
\begin{theorem}\label{thm:presentation WN full}
The following isomorphism of stacks holds over $\spec(\bZ[\frac{1}{6}])$:
\[ \cW^{\min}_N\simeq [(V^N_{4N,6N}\smallsetminus\Delta_N)/(\GL_2/\bmu_N)]. \]
\end{theorem}
\begin{Remark}\label{rmk:two presentations}
The presentation above specializes to the two following cases depending on the parity of $N$, that is:
\begin{itemize}
    \item if $N$ is odd, then $\cW^{\min}_N\simeq [(V^N_{4N,6N}\smallsetminus\Delta_N)/\GL_2]$;
    \item if $N$ is even, then $\cW^{\min}_N\simeq [(V^N_{4N,6N}\smallsetminus\Delta_N)/\PGL_2\times \bG_m]$.
\end{itemize}
The actions of these two groups are the ones explained in the Notation section.
\end{Remark}

\subsection{Vector bundles on $\cW^{\min}_N$ when $N$ is odd}
Let us suppose $N$ odd. As oberved in \Cref{rmk:two presentations}, the stack $\cW_N^{\min}$ has a presentation as a quotient by the action of $\GL_2$. In particular, the $\GL_2$-equivariant morphism $V^N_{4N,6N}\smallsetminus\Delta_N\to\spec(\bZ[\frac{1}{6}])$ induces a morphism of quotient stacks $\cW_N^{\min}\to\cB\GL_2$.
This should correspond to a rank two vector bundle on $\cW_N^{\min}$.

\begin{Def}\label{def:vec N odd}
For $N$ odd, we define the rank two vector bundle $\cE_N$ on $\cW^{\min}_N$ as follows:
\[ \cE_N((\cX,\cS)\overset{f}{\to} \cP \overset{p}{\to} T):=p_*((R^1f_*\cO)^{\vee}\otimes\omega_{\cP/T}^{\otimes \frac{N-1}{2}}). \]
\end{Def}
\begin{Prop}\label{prop:class map odd}
The map $\cW_N^{\min}\to \cB\GL_2$ is given by
\[ ((\cX,\cS)\overset{f}{\to}\cP\overset{p}{\to} T)\longmapsto \cE_N((\cX,\cS)\overset{f}{\to}\cP\overset{p}{\to} T). \]
\end{Prop}
\begin{proof}
First we claim that the isomorphism $\cB(\GL_2/\bmu_N) \simeq \cB(\GL_2)$ sends a pair $(\cP\overset{p}{\to} T,\cL)$ to $p_*(\cL\otimes \omega_{\cP/T}^{\otimes \frac{N-1}{2}})$. Indeed, consider first the homomorphism $\GL_2\to \GL_2$ that sends $A$ to $\det(A)^{\frac{N-1}{2}}A$; this descends to the isomorphism $\GL_2/\bmu_N \to \GL_2$.

The induced morphism $\cB\GL_2\to\cB\GL_2$ sends a rank two vector bundle $E\to T$ to $\det(E)^{\otimes \frac{N-1}{2}}\otimes E$. On the other hand, the morphism $\cB\GL_2\to \cB(\GL_2/\bmu_N)$ sends $E$ to $(\bP(E^\vee),\cO_{\bP(E^\vee)}(N))$, from which we deduce that $\cB(\GL_2/\bmu_N)\to\cB\GL_2$ sends $(\bP(E^\vee),\cL=\cO_{\bP(E^\vee)}(N))$ to $\det(E)^{\otimes \frac{N-1}{2}}\otimes E$: with a straightforward computation involving the Euler short exact sequence on $\bP(E^\vee)$, we see that the latter vector bundle is isomorphic to $p_*(\cL\otimes\omega_{\bP(E^\vee)/T}^{\otimes \frac{N-1}{2}})$. The claimed description follows then by descent.

%The stack $\cB\GL_2$ is isomorphic to the stack $\cY$ whose objects are pairs $(\cP\to T,\cN)$, where $\cP\to T$ is a Severi-Brauer variety whose geometric fibers have dimension one, and $\cN$ is a line bundle on $\cP$ that on the geometric fibers has degree one: indeed, given a vector bundle of rank $2$ $E\to T$, one can take $\cP=\bP(E^\vee)$ and $\cN=\cO_{\bP(E^\vee)}(1)$ (the latter is canonical once $E$ is fixed); viceversa, given $(\pi:\cP\to T,\cN)$, we have that $\pi_*\cN$ is a rank two bundle over $T$. It is immediate to check that the two constructions above are one the quasi-inverse of the other. 
%Therefore, to give a map to $\cB\GL_2$ is equivalent to giving such a pair $(\cP\to T,\cN)$,
%On the other hand, the stack $\cB\GL_2$ can also be regarded as the stack of rank two vector bundles. The equivalence between these two descriptions is given in one direction by sending a pair $(p:\cP\to T,\cN)$ to the rank two vector bundle $p_*\cN$, and in the other direction by sending the vector bundle $\cE$ on $T$ to $(\bP(\cE^{\vee}),\cO_{\cP(\cE^{\vee})}(1))$.

By construction, the map $\cW_{\min}^N\to\cB(\GL_2/\bmu_N)$ is as follows:
\[((\cX,\cS)\overset{f}{\to} \cP \overset{p}{\to} T)\longmapsto(\cP\to B,(R^1f_*\cO)^{\vee}).\]
The composition $\cW_{\min}^N \to \cB(\GL_2/\bmu_N) \to \cB\GL_2$ corresponds then to $p_*((R^1f_*\cO)^{\vee}\otimes\omega_{\cP/T}^{\otimes \frac{N-1}{2}})$.
\end{proof}
\subsection{Vector bundles on $\cW^{\min}_N$ when $N$ is even}
In this case, the presentation of the stack $\cW_N^{\min}$ given in \Cref{thm:presentation WN full} can be recasted in terms of the group $\PGL_2\times \bG_m$. In particular, this shows that there is a map $\cW_N^{\min}\to \cB\PGL_2\times\cB\bG_m$, which then must be induced by a Severi-Brauer stack on $\cW_N^{\min}$ together with a line bundle.

\begin{Def}\label{def:vec N even}
For $N$ even, we define the rank three vector bundle $\cE_N$ and the line bundle $\cL_N$ on $\cW^{\min}_N$ as follows:
\begin{align*}
    &\cE_N((\cX,\cS)\overset{f}{\to}\cP\overset{p}{\to} T):=p_*(\omega_{\cP/T}^{\vee}),\\
    &\cL_N((\cX,\cS)\overset{f}{\to}\cP\overset{p}{\to} T):=p_*((R^1f_*\cO_{\cX})^{\vee}\otimes \omega_{\cP/T}^{\otimes \frac{N}{2}}).
\end{align*}
\end{Def}

The vector bundle $\cE_N$ actually plays no role here, but it will be relevant later on.

\begin{Prop}\label{prop:class map even}
The map $\cW_N^{\min}\to \cB\PGL_2\times\cB\bG_m$ is given by
\[ ((\cX,\cS)\overset{f}{\to}\cP\overset{p}{\to} T)\longmapsto  (\cP\overset{p}{\to}T,\cL_N((\cX,\cS)\overset{f}{\to}\cP\overset{p}{\to} T))\]
\end{Prop}

\begin{proof}
The stack $\cB(\GL_2/\bmu_N)$ classifies pairs $(\cP\overset{p}{\to} T,\cL)$ where $\cP\to T$ is a Severi-Brauer variety and $\cL$ is a line bundle on $\cP$ whose restriction to the geometric fibers of $\cP\to T$ has degree $N$.
We claim that the isomorphism $\cB (\GL_2/\bmu_N)\simeq \cB(\PGL_2\times\bG_m)$ sends a pair $(\cP\overset{p}{\to} T,\cL)$ to the pairs $(\cP\overset{p}{\to} T,p_*(\cL\otimes\omega_{\cP/T}^{\otimes \frac{N}{2}}))$. 

Indeed, consider the homomorphism $\GL_2\to\PGL_2\times\bG_m$ which sends $A\mapsto ([A],\det(A)^{\frac{N}{2}})$: this homomorphism descends to the isomorphism $\GL_2/\bmu_N \overset{\simeq}{\to} \PGL_2\times\bG_m$. The induced morphism $\cB\GL_2\to\cB(\PGL_2\times\bG_m)$ then works as follows: a rank two vector bundle $E\to T$ with associated cocycles $\{A_{ij}\}$ is sent to the torsor whose associated cocycles are $\{[A_{ij}],\det(A_{ij})^{\frac{N}{2}}\}$, i.e. the object $(\bP(E^\vee)\to T,\det(E)^{\otimes \frac{N}{2}})$.

Observe now that the morphism $\cB\GL_2 \to \cB(\GL_2/\bmu_N)$ sends a rank two vector bundle $E\to T$ to $(\bP(E^\vee)\to T, \cL=\cO_{\bP(E^\vee)}(N))$; we deduce that $\cB(\GL_2/\bmu_N)\to\cB(\PGL_2\times\bG_m)$ must send $(\bP(E^\vee)\overset{p}{\to} T, \cL=\cO_{\bP(E^\vee)}(N))$ to $(\bP(E^\vee)\to T,\det(E)^{\otimes \frac{N}{2}})$. A straightforward computation with the Euler sequence of $\bP(E^\vee)\overset{p}{\to} T$ shows that $\det(E)^{\otimes \frac{N}{2}} \simeq p_*(\cL\otimes\omega_{\bP(E^\vee)/T}^{\otimes \frac{N}{2}})$.

As every object $(\cP\to T,\cL)$ is \'{e}tale locally isomorphic to an object of the form $(\bP(E^\vee)\to T, \cO_{\bP(E^\vee)}(N))$, we obtain by descent the claimed description of the isomorphism $\cB (\GL_2/\bmu_N)\simeq \cB(\PGL_2\times\bG_m)$.

The map $\cW_N^{\min}\to \cB\PGL_2\times\cB\bG_m$ can be factored as
\[ \cW_N^{\min} \longrightarrow \sF^{\min}_N \longrightarrow \cB(\GL_2/\bmu_N) \longrightarrow \cB(\PGL_2\times\bG_m),\]
where the composition of the first two maps sends an object $((\cX,\cS)\overset{f}{\to}\cP\overset{p}{\to} T)$ to the pair $(\cP\overset{p}{\to}T,(R^1f_*\cO_{\cX})^{\vee})$, from which we deduce that the composition sends a family of minimal Weierstrass fibrations over $T$ to the pair \[(\cP\overset{p}{\to}T,p_*((R^1f_*\cO_{\cX})^{\vee}\otimes \omega_{\cP/T}^{\otimes \frac{N}{2}})).\]
\end{proof}
\section{Equivariant intersection theory on projective spaces}\label{equivint}
in this section, we work over a ground field of any characteristic.
Set $V_k:=\oH^0(\bP^1,\cO(k))$. The projective space $\bP V_k$ can be naturally identified with the Hilbert scheme of $k$ points on $\bP^1$. In this section we consider two actions on $\bP V_k$, namely:
\begin{itemize}
    \item the $\PGL_2$-action inherited from the natural action of $\PGL_2$ on $\bP^1$, that is $[A]\cdot [f(x,y)]:=[f(A^{-1}(x,y))]$;
    \item the $\GL_2$-action induced by the $\PGL_2$-action above via the homomorphism $\GL_2\to\PGL_2$.    
\end{itemize}

The aim of this Section is to collect some basic facts on the integral Chow ring of $[\bP V_k/G]$, where $G$ is either $\GL_2$ or $\PGL_2$. We will divide our analysis in two parts, depending on whether $V_k$ is a $G$-representation or not. The reason for this is that, given a projective linear action of a group $G$ over a projective space $\bP(V)$, the resulting quotient stack $[\bP(V)/G]$ is a projective bundle over $\cB G$ if and only if the action of $G$ lifts to a linear action on $V$, i.e. if $V$ is a $G$-representation.

\subsection{First case}
As $V_k$ is a $\GL_2$-representation, the stack $[\bP V_k/\GL_2]$ is a projective bundle over $\cB\GL_2$. Similarly, for $k$ even, the vector space $V_k$ is a $\PGL_2$-representation, where the action is defined as
\[ A\cdot f(x,y) := \det(A)^{\frac{k}{2}}f(A^{-1}(x,y)). \]
Equivalently, the representation above is obtained by taking the $\GL_2/\bmu_k$-representation $V_k^k$ of Section \Cref{subsec:basic} and endowing it with a $\PGL_2$ action via the homomorphism $\PGL_2 \to \GL_2/\bmu_k$ defined as $[A]\mapsto [\det(A)^{-\frac{1}{2}} A]$.
Therefore, for $G=\GL_2$ or $G=\PGL_2$ and $k$ even, we have that $\pi:[\bP V_k/G]\to\cB G$ is a projective bundle.

Let $h$ be the hyperplane class. From an equivariant point of view, we can regard $h$ as the class of the $G$-equivariant line bundle $\cO_{\bP V_k}(1)$. The following Proposition is just the usual projective bundle formula.
\begin{Prop}\label{prop:chow hilb even}
Assume that either $G=\GL_2$ or $G=\PGL_2$ and $k$ is even. Then:
\begin{enumerate}
    \item The integral Chow ring of $[\bP V_k/G]$ is generated as $\ch^*(\cB G)$-module by $h^m$ for $m\leq k$.
    \item We have $\pi_*(h^m)=s_{m-k}^{G}(V_k)$, where the latter denotes the $G$-equivariant Segre class of degree $m-k$ of $V_k$.
\end{enumerate}
\end{Prop}

\subsection{Second case}\label{second case}
For $k$ odd, the vector space $V_k$ is not a $\PGL_2$-representation: indeed, any lift of the $\PGL_2$-action on $\bP(V_k)$ to $V_k$ should be of the form $[A]\cdot f(x,y)=\det(A)^d f(A^{-1}(x,y))$, and there is no choice of $d$ which makes the formula above well defined; picking a different representative $\lambda A$ for $[A]$ makes the right hand side equal to $\lambda^{2d-k}\det(A)^d f(A^{-1}(x,y))$. This implies that the quotient stack $[\bP V_k/\PGL_2]$ is not a projective bundle over $\cB\PGL_2$. We have to treat this second case differently.

Let $\Sigma_k\subset \bP V_k\times\bP^1$ be the $\PGL_2$-invariant subscheme defined as
\[ \Sigma_k=\{ (f,x) \text{ such that } f(x)=0 \}. \]
The line bundle $\cO(\Sigma_k)$ is isomorphic to $\pr_1^*\cO_{\bP V_k}(1)\otimes\pr_2^*\cO_{\bP^1}(k)$, hence the latter admits a $\PGL_2$-linearization. The canonical line bundle $\omega_{\bP^1}\simeq\cO_{\bP^1}(-2)$ admits a $\PGL_2$-linearization as well. Then the isomorphism
\begin{equation*}
    \cO_{\bP V_k}(1)\otimes V_1\simeq \pr_{1*}(\pr_1^*\cO_{\bP V_k}(1)\otimes\pr_2^*\cO_{\bP^1}(k)\otimes\pr_2^*\omega_{\bP^1}^{\frac{k-1}{2}}),
\end{equation*}
gives a $\PGL_2$-linearization to the rank two vector bundle $\cO_{\bP V_k}(1)\otimes V_1$.

Before going on, recall (\cite{Phan}) that $\ch^*(\cB\PGL_2)\simeq\bZ[c_2,c_3]/(2c_3)$, where $c_i$ is the $i^{\rm th}$ Chern class of the vector bundle $[V_2/\PGL_2]\to\cB\PGL_2$. In what follows we will use the following standard convention for Chern classes: as the projection morphism $\pi:[\bP V_k/\PGL_2]\to\cB\PGL_2$ gives to $\ch^*([\bP V_k/\PGL_2])$ the structure of a $\ch^*(\cB\PGL_2)$-module via the pullback homomorphism, we will use the notation $c_i$ for the classes $\pi^*c_i$. With this convention, the projection formula $\pi_*(\pi^*c_i^d \cdot \eta)=c_i^d\cdot \pi_*\eta$ simply reads as $\pi_*(c_i^d \cdot \eta) = c_i^d \cdot \pi_*\eta$. 
\begin{Prop}\label{prop:chow hilb odd}
    For $k\geq 0$ odd, let $\gamma_1$, $\gamma_2$ be the $\PGL_2$-equivariant Chern classes of $\cO_{\bP V_k}(1)\otimes V_1$. Then we have $$\ch^*([\bP V_k/\PGL_2])\simeq \bZ[\gamma_1,\gamma_2]/(\prod_{i=0}^{\frac{k-1}{2}} \left(\gamma_2+\frac{(k-2i)^2-1}{4}c_2\right)),$$
    where $c_2=-\gamma_1^2+4\gamma_2$, and the generators of this ring as a module over $\ch^*(\cB\PGL_2)$ are $\gamma_1^i\gamma_2^j$, where $i=0,1$ and $j=0,1,\dots,\frac{k-1}{2}$.
\end{Prop}
\begin{proof}
For $k$ odd, in \cite[Proposition 3.7]{ST21} it is proved that the Chow ring of $\ch^*([\bP V_k/\PGL_2])$ is isomorphic to 
\[\bZ[u,v]^{S_2}/(\prod_{i=0}^{k} ((\frac{k+1}{2}-i)u + (\frac{-k+1}{2}+i)v)),\]
where $u,v$ are the Chern roots of $\cO_{\bP V_k}(1)\otimes V_1$, so that $u+v=\gamma_1$ and $uv=\gamma_2$. 

Moreover, we know from \cite{ST21} that the pullback homomorphism $\ch^*(\cB \PGL_2)\to \ch^*([\bP V_k/\PGL_2])$ sends $c_2\mapsto -(u-v)^2$ and $c_3\mapsto 0$. This implies that
\[ \gamma_1^2=(u+v)^2 = (u-v)^2+4uv = 4uv-c_2=4\gamma_2-c_2. \]
In particular $\ch^*([\bP V_k/\PGL_2])$ is generated as a module by monomials of the form $(u+v)^i(uv)^j$, where $i$ is either $0$ or $1$.

We can then rewrite the relation as follows:
\begin{align*}
    %\prod_{i=0}^{n} \left(\left(\left(\frac{n+1}{2}-i\right)u + \left(\frac{-n+1}{2}+i\right)v\right)\right) 
    & \prod_{i=0}^{\frac{k-1}{2}} \left(\left(\frac{k+1}{2}-i\right)u + \left(\frac{-k+1}{2}+i\right)v\right)\left(\left(\frac{-k+1}{2}+i\right)u + \left(\frac{k+1}{2}-i)v\right) \right)\\
    =& \prod_{i=0}^{\frac{k-1}{2}} \left(uv+\left(\frac{k+1}{2}-i\right)\left(\frac{-k+1}{2}+i\right)(u-v)^2\right) \\
    =&\prod_{i=0}^{\frac{k-1}{2}} \left(uv+\frac{(k-2i)^2-1}{4}c_2\right).
\end{align*}
This shows that the monomials $(u+v)^i(uv)^j$ for $i\leq 1$ and $j\leq \frac{k-1}{2}$ actually generate $\ch^*([\bP V_k/\PGL_2])$ as a module over $\ch^*(\cB\PGL_2)$, as claimed.
\end{proof}
\begin{Remark}\label{oss:P1}
For $k=1$, we can see that the Chow ring of $[\bP V_1/\PGL_2]$ is generated by the first Chern class of the normal bundle of the universal section which, coherently with the usual definition of psi classes, we denote $\psi_1$. In fact, on the universal conic $p:\cP\to [\bP V_1/\PGL_2]$ we have a short exact sequence
\begin{equation} \label{eq:sequence}
    0\longrightarrow \cO_{\cP} \longrightarrow \cO(\sigma) \longrightarrow \sigma_*\sigma^*\cO(\sigma) \longrightarrow 0
\end{equation}
where $\sigma$ is the universal section.
By pushing forward along $p$, we get an exact sequence of locally free sheaves
\[0\longrightarrow \cO \longrightarrow p_*\cO(\sigma) \longrightarrow \sigma^*\cO(\sigma) \longrightarrow 0 \]
which shows that the rank two bundle in the middle is an extension of the normal bundle of the universal section by the trivial line bundle. This implies that $\gamma_1=\psi_1$ and $\gamma_2=0$.
\end{Remark}
Next we give an explicit description of the pushforward morphism along $\pi:[\bP V_k/\PGL_2]\to \cB\PGL_2$. For this, set
$$E_{n,m}(q):=(-1)^q\underset{a+b=2q+1}{\sum_{a=0}^{m}\sum_{b=0}^{n}}2^{m-a}\binom{m}{a}\binom{n}{b}.$$
\begin{Lemma}\label{lm:push odd}
    We have $$\pi_*(\gamma_1^m\gamma_2^n)=k^{-1}\sum_{0\leq q\leq n+\frac{m-k}{2}} E_{n,m}(q)\cdot s^{\PGL_2}_{2(n-q)+m-k}(V_{k-1})\cdot 2c_2^q.$$
\end{Lemma}
Observe that the sum above is actually a scalar multiple of $c_2$: this is because every monomial containing $c_3$ that appears in a Segre class is killed by the multiplication by $2$, hence the polynomial above lives in the ring $\bZ[c_2]$. In this way the multiplication by the inverse of $k$ can be understood literally, i.e. as the division of the scalar coefficient by $k$.

In particular, we are implying that such scalar coefficient is a multiple of $k$, because the whole expression belongs to the integral Chow ring.
\begin{proof}
Consider the commutative diagram of quotient stacks
\begin{equation}\label{eq:diag}
    \begin{tikzcd}
    \left[\bP V_1 \times \bP V_{k-1}/\PGL_2\right] \arrow[d, "\pr_1"] \arrow[r, "\rho"] & \left[\bP V_k/\PGL_2\right] \arrow[d, "\pi"] \\
    \left[\bP V_1/\PGL_2\right] \arrow[r, "\pi'"] & \cB \PGL_2
\end{tikzcd}
\end{equation}
where the top horizontal arrow is induced by the multiplication map. This map is finite of degree $k$, hence $\rho_*\rho^*\xi=k\xi$ and
\begin{equation*}
    k\cdot \pi_*\xi = \pi_*\rho_*\rho^*\xi = \pi'_*\pr_{1*}(\rho^*\xi).
\end{equation*}
As $k$ is odd, multiplication by $k$ is an injective group endomorphism of $\ch^*(\cB \PGL_2)$. This argument shows that once we understand how the pullback homomorphism $\rho^*$ and the composition $\pi'_*\pr_{1*}$ works, we also have an explicit formula for $\pi_*$.

We first have to compute the pullback of $\gamma_1$ and $\gamma_2$ to the Chow ring of $\left[\bP V_1 \times \bP V_{k-1}/\PGL_2\right]$.  For this, observe that $\rho^*(V_1\otimes\cO_{\bP V_k}(1))=V_1\otimes\pr_1^*\cO_{\bP V_1}(1) \otimes\pr_2^*\cO_{\bP V_{k-1}}(1)$.
Recall from \Cref{prop:chow hilb even} and \Cref{oss:P1} that $h=c_1^{\PGL_2}(\cO_{\bP V_{k-1}}(1))$ and $\psi_1=c_1^{\PGL_2}(V_1\otimes\cO_{\bP V_1}(1))$. Applying the splitting principle and the additivity of the total Chern class, we deduce
\begin{align*}
    &\rho^*\gamma_1=c_1^{\PGL_2}(V_1\otimes\pr_1^*\cO_{\bP V_1}(1)\otimes \pr_2^*\cO_{\bP V_{k-1}}(1))=\psi_1+2h\\
    &\rho^*\gamma_2=c_2^{\PGL_2}(V_1\otimes\pr_1^*\cO_{\bP V_1}(1)\otimes \pr_2^*\cO_{\bP V_{k-1}}(1))=h(h+\psi_1).
\end{align*}
This implies that
\begin{align*}
    \pi_*(\gamma_1^m\gamma_2^n)&=k^{-1}\pi'_*\pr_{1*}((2h+\psi_1)^mh^n(h+\psi_1)^n).
\end{align*}
The computation of pushforwards along $\pr_1:[\bP V_1 \times \bP V_{k-1}/\PGL_2]\to [\bP V_1/\PGL_2]$ is easy because $k-1$ is even, hence this map is the projection from a projective bundle. We deduce
\begin{equation*}
    \pr_{1*}(h^i\psi_1^j)=s^{\PGL_2}_{i-k+1}(V_{k-1})\psi_1^j.
\end{equation*}
Also the pushforward along $\pi':[\bP V_1/\PGL_2]\to\cB\PGL_2$ is not hard to determine: consider the cartesian diagram
\[
\xymatrix{
\bP V_1 \ar[r]^{\rho'} \ar[d]^{g} & \spec{k} \ar[d]^{f} \\
[\bP V_1/\PGL_2] \ar[r]^{\pi'} & \cB\PGL_2.
}
\]
The compatibility formula implies that for every element $\xi$ in the Chow ring of $[\bP V_1/\PGL_2]$ we have $f^*\pi'_*(\xi)=\rho'_*g^*(\xi)$. To compute $g^*(\psi_1)$, observe that $g^*(V_1\otimes\cO_{\bP V_1}(1))=\cO_{\bP V_1}(1)^{\oplus 2}$, hence $g^*(\psi_1)=2c_1^{\PGL_2}(\cO_{\bP V_1}(1))$; this implies $f^*\pi'_*(\psi_1)=\rho'_*g^*(\psi_1)=2$.
In degree zero the pullback $f^*:\ch^0(\cB\PGL_2)\to\ch^0(\spec{k})$ is an isomorphism, so we can conclude $\pi'_*\psi_1=2$.

The relation $\psi_1^2=-c_2$ implies that $\pi'_{*}(\psi_1^{2j})=0$ and $\pi'_{*}(\psi_1^{2j+1})=(-1)^j 2c_2^j$.
We deduce 
\begin{align*}
   \pi'_*\pr_{1*}( h^i\psi_1^{2j} )=0,\quad  \pi'_*\pr_{1*}( h^i\psi_1^{2j+1} )=(-1)^j s^{\PGL_2}_{i-k+1}(V_{k-1})2c_2^j.
\end{align*}
Putting all together, we obtain the claimed formulas for the pushforward along $\pi:[\bP V_k/\PGL_2]\to\cB\PGL_2$.
\end{proof}
%\begin{Cor}
%Something about the Chow ring of $[\bP^n/\PGL_2\times \bG_m]$ and $[\bP^n/\GL_2]$.
%\end{Cor}

\subsection{Chern classes of representations}
Here we outline how to explicitly compute the Chern classes of the representations that appeared before. This also gives formulas for the Segre classes by formally inverting the total Chern class.

First, let us consider the case $G=\GL_2$. The integral Chow ring of $\cB\GL_2$ is isomorphic to $\bZ[c_1,c_2]$, where $c_1$ and $c_2$ are the Chern classes of the standard $\GL_2$-representation $E$. Therefore, if $\ell_1$ and $\ell_2$ are the Chern roots of $E^{\vee}$, we have that $c_1=-(\ell_1+\ell_2)$ and $c_2=\ell_1\ell_2$.

We have $V_m=\Sym^{m} E^{\vee}$, hence the Chern roots of this symmetric power are given by $j\ell_1+(m-j)\ell_2$, where $0\leq j\leq m$. From this we deduce that the total Chern class of $V_{m}$ for $m$ even is equal to
\begin{align*}
    c^{\GL_2}(V_m)&=\prod_{j=0}^{m}(1+(j\ell_1+(m-j)\ell_2)t)\\
    &=(1+\frac{m}{2}(\ell_1+\ell_2)t)\prod_{j< \frac{m}{2}} (1+(j\ell_1+(m-j)\ell_2)t)(1+((m-j)\ell_1+j\ell_2)t)\\
    &=(1-\frac{m}{2}c_1t)\prod_{j< \frac{m}{2}} (1-mc_1t+(j(m-j)c_1^2+(2j-m)^2c_2)t^2).
\end{align*}
whether for $m$ odd the same argument gives us
\[ c^{\GL_2}(V_m)=\prod_{j\leq \frac{m}{2}} (1-mc_1t+(j(m-j)c_1^2+(2j-m)^2c_2)t^2). \]
Let $\langle p(t)\rangle_d$ denote the coefficient in front of $t^d$ in $p(t)$. Then we have proved the following:
\begin{Prop}\label{prop:chern V2m 1}
\begin{align*}
     c_d^{\GL_2}(V_m)=
     \left\{
     \begin{array}{ll}
         \left \langle (1-\frac{m}{2}c_1t)\prod_{j< \frac{m}{2}} (1-mc_1t+(j(m-j)c_1^2+(2j-m)^2c_2)t^2) \right \rangle_d & \text{if }m\text{ is even,}\\
         \\
         \left \langle \prod_{j\leq \frac{m}{2}} (1-mc_1t+(j(m-j)c_1^2+(2j-m)^2c_2)t^2) \right \rangle_d & \text{if }m\text{ is odd.}
     \end{array}
     \right.
\end{align*}
\end{Prop}

Next, we consider the case $G=\PGL_2$.
The vector space $V_{2m}=\oH^0(\bP^1,\cO(2m))$ is a $\PGL_2$-representation of rank $2m+1$, where the action is defined as $A\cdot f(x,y):=\det(A)^{m}f(A^{-1}(x,y)$. In what follows we will need explicit formulas for the $\PGL_2$-equivariant Chern classes of $V_{2m}$. These has been computed by Fulghesu and Viviani in \cite{FV}*{Section 6}.

Recall  (\cite{Phan}) that $\ch^*(\cB \PGL_2)$ is isomorphic to $\bZ[c_2,c_3]/(2c_3)$, where $c_i^{\PGL_2}(V_2)=c_i$, and
\begin{Prop}[\cite{FV}*{Corollary 6.3}]\label{prop:chern V2m 2}
\[
    c_d^{\PGL_2}(V_{2m})=\left \{ \begin{array}{ll}
\left\langle t\prod_{j=1}^{m} (t^2+j^2c_2) + t^{\frac{m}{2}+1}\sum_{j=1}^{\frac{m}{2}}\binom{\frac{m}{2}}{j}(t^3+c_2t)^{(\frac{m}{2}-j)}c_3^j \right\rangle_{2m+1-d} & \text{if }m\text{ is even,} \\
\\
\left\langle t\prod_{j=1}^{m} (t^2+j^2c_2) + t^{\frac{m-1}{2}}\sum_{j=1}^{\frac{m+1}{2}}\binom{\frac{m+1}{2}}{j}(t^3+c_2t)^{(\frac{m+1}{2}-j)}c_3^j \right\rangle_{2m+1-d} & \text{if }m\text{ is odd.}
\end{array}
\right.
\]
\end{Prop}
The $\PGL_2$-equivariant Segre classes of $V_{2m}$ can then be computed by formally inverting the total Chern classes of the $\PGL_2$-representation.

\section{Relations coming from $[\Delta^1_N/G_N]$}\label{Delta1}
In this Section, we compute relations in the Chow ring of $\cW^{\min}_N$ obtained excising $[\Delta^1_N/G_N]$. More precisely, we show that the ideal of relations obtained by excising this locus has a single generator (\Cref{envelope}) and we give a recipe for computing it (see \Cref{prop:class D1 GL2} and \Cref{prop:class D1 PGL}). We work over a ground field of characteristic different from $2$ or $3$.

\subsection{Excision of $[\Delta_N^1/G_N]$}
Consider the localization exact sequence
\[
\ch^{G_N}_*(\Delta^1_N)\rightarrow \ch^{G_N}_*(V^N_{4N,6N})\rightarrow \ch^{G_N}_*(V^N_{4N,6N}\smallsetminus \Delta^1_N)\rightarrow 0.
\] 
We want to find generators for the ideal given by the image of the first map on the left. To do so, we construct an equivariant envelope of $\Delta^1_N$, in the sense of \cite[Page 603]{EG}, i.e. a proper morphism $Z\to\Delta^1_N$ whose induced pushforward homomorphism between Chow groups is surjective. 
Let 
\begin{equation}\label{eq:map phi}
   \phi:V^N_{2N}\longrightarrow V^N_{4N,6N}
\end{equation}
be the map defined by $\phi(P)=(-3P^2,2P^3)$. Observe that the image of $\phi$ lies in $\Delta_N^1$.
\begin{Lemma}\label{envelope}
The following hold true:
\begin{enumerate}
    \item the map $\phi$ defines a $G_N$-equivariant bijective birational morphism $V^N_{2N}\to\Delta_N^1$ that is an isomorphism away from the origin;
    \item the pushforward morphism $\ch^{*}_{G_N}(V^N_{2N})\to \ch^*_{G_N}(\Delta_N^1)$ is surjective;
    \item the image of $\phi_*$ is the ideal generated by $[\Delta_N^1]_{G_N}$.
\end{enumerate}
\end{Lemma}
\begin{proof}
Away from the origin, the map $(A,B)\mapsto -3B/2A$ defines an equivariant inverse to $V^N_{2N}\to\Delta_N^1$, so $\phi$ is bijective and is an isomorphism away from the origin. This also implies the surjectivity of the induced pushforward.

To prove the last point, observe that there is a well defined pullback morphism $\phi^*$ because $V^N_{4N,6N}$ is smooth, and $\phi^*$ is clearly surjective because both $[V^N_{2N}/G_N]$ and $[V^N_{4N,6N}/{G_N}]$ are vector bundles over $\cB {G_N}$. Therefore, for every cycle $\zeta$ in $\ch^*_{G_N}(V^N_{2N})$, we have $\phi_*(\zeta)=\phi_*\phi^*(\zeta')=\phi_*(1)\cdot\zeta'$. This proves the last point.
\end{proof}
\subsection{The case ${G_N}=\GL_2$}
When $N$ is odd, the group ${G_N}$ is isomorphic to $\GL_2$. To compute $[\Delta_N^1]_{\GL_2}$ we can apply the localization formula (\cite[Theorem 2]{EGloc}). In general, this formula only gives an expression which is true up to cycles that are zero divisor. In our case we are lucky, as the equivariant Chow ring of $V^N_{4N,6N}$ is a polynomial ring in the two variables $c_1$ and $c_2$, so the expression we obtain in the end holds true unconditionally.
\begin{Prop}\label{prop:class D1 GL2}
For $N$ odd, the image of $\ch_*^{G_N}(\Delta_N^1)\to\ch_*^{G_N}(V^N_{4N,6N})$ is generated as an ideal by $$[\Delta_N^1]_{G_N}=\frac{c_{10N+2}^{G_N}(V^N_{4N,6N})}{c_{2N+1}^{G_N}(V^N_{2N})}.$$
\end{Prop}
\begin{proof}
Let $T\subset \GL_2$ be the maximal subtorus of diagonal matrices. The point in $V^N_{rN}$ fixed by the $T$-action is the origin, whose tangent space is isomorphic to $V^N_{rN}$ itself. Applying localization formula (\cite[Theorem 2]{EGloc}), we deduce that
\begin{align*}
    \phi_*(1)=\frac{\phi_*([(0)]_T)}{c_{2N+1}^T(TV^N_{2N,(0)})}=\frac{[(0)]_T}{c_{2N+1}^T(TV^N_{2N,(0)})}=\frac{c_{10N+2}^T(V^N_{4N,6N})}{c_{2N+1}^T(V^N_{2N})}.
\end{align*}
As the $T$-equivariant Chern classes of a $\GL_2$-equivariant vector bundle are equal to the $\GL_2$-equivariant ones, we obtain an expression for $[\Delta_N^1]_{G_N}$. By \Cref{envelope}, this class generates the ideal $\im(\phi_*)$, and from the same Lemma we know that this ideal coincides with the image of $\ch^{*-8N-1}_{G_N}(\Delta_N^1)\to \ch^*_{G_N}(V^N_{4N,6N})$.
\end{proof}

\subsection{The case ${G_N}=\PGL_2\times\bG_m$}
For $N$ even, the group that we have to consider is $\PGL_2\times\bG_m$. Because of the fact that in $\ch^*(\cB(\PGL_2\times\bG_m))$ there are zero divisors, e.g. the integer $2$, we cannot apply the localization formula directly. To overcome this obstacle, we will use a trick introduced in \cite{DL}.

Let $f:X'\to X$ be a $\PGL_2\times\bG_m$-equivariant morphism between $\PGL_2\times\bG_m$-equivariant schemes. Then by \cite[Theorem 2.11]{DL} there exist $\GL_3\times\bG_m$-schemes $Y$ and $Y'$ with an equivariant morphism $Y'\to Y$ and a commutative diagram
\begin{equation}
\begin{tikzcd}
\left[ X'/\PGL_2\times\bG_m \right] \ar[r,"\simeq"] \ar[d] & \left[ Y'/\GL_3\times\bG_m \right] \ar[d] \\
\left[ X/\PGL_2\times\bG_m \right]  \ar[r,"\simeq"] & \left[ Y/\GL_3\times\bG_m \right].
\end{tikzcd}
\end{equation}
We refer to the $\GL_3\times\bG_m$-scheme $Y$ (resp. $Y'$) as the $\GL_3\times\bG_m$-counterpart of $X$ (resp. $X'$), as in \cite{DL}*{Definition 2.9}. Recall that $V^N_{2dN}$ is the $\PGL_2\times\bG_m$-representation $\oH^0(\bP^1,\cO(2dN))\otimes L^{\otimes(-2d)}$, where $L$ is the standard rank one representation of $\bG_m$: our aim is to describe explicitly the $\GL_3\times\bG_m$-counterpart of $V^N_{2dN}$.

The affine space $\bA^6$ is the parameter space of quadratic forms in three variables, and let $D$ be the discriminant divisor, i.e. the divisor that parametrizes quadratic forms of rank $\leq 2$. We regard $\bA^6$ as a $\GL_3\times\bG_m$-scheme, where $\bG_m$ acts trivially and $\GL_3$ acts as $A\cdot q(x,y,z)=\det(A)q(A^{-1}(x,y,z))$. Observe that $D$ is invariant with respect to this action.

Over $\bA^6\smallsetminus\{0\}$ we have an injective morphism of $\GL_3\times\bG_m$-equivariant free sheaves
\begin{equation}\label{eq:def of W}
    \oH^0(\bP^2,\cO(dm-2))\otimes L^{\otimes(-2d)}\otimes\cO_{\bA^6\smallsetminus\{0\}}\longrightarrow \oH^0(\bP^2,\cO(dm))\otimes L^{\otimes(-2d)}\otimes\cO_{\bA^6\smallsetminus\{0\}},
\end{equation}
where the $\GL_3$-action on these sheaves is inherited from the natural action of $\GL_3$ on $\bP^2$, the latter regarded as the projectivization of the standard $\GL_3$-representation.

The quotient of the map in (\ref{eq:def of W}) is denoted $W^m_{2dm}$ and by \cite{DL}*{Proposition 3.4} the restriction of $W^m_{2dm}$ to $\bA^6\smallsetminus D$ is the $\GL_3\times\bG_m$-counterpart of $V^m_{2dm}$. Moreover, we adopt the notation $W^m_{2dm,2em}$ for the direct sum of $W^m_{2dm}$ and $W^m_{2em}$. The $K$-points in the total space of $W^m_{2dm}$ should be thought as pairs $(q,[f])$ where $q$ is a non-zero ternary quadratic form on $K$, the polynomial $f$ is a homogeneous form in three variables of degree $dm$ and $[f]=[f']$ if and only if $q$ divides the difference $f-f'$.

In this way we can also describe the counterpart of the equivariant map $\phi:V^N_{2N}\to V^N_{4N,6N}$ introduced in (\ref{eq:map phi}), which is the restriction to $\bA^6\smallsetminus D$ of the morphism
\[ \psi: W^N_{2N}\longrightarrow W^N_{4N,6N},\quad (q,[f])\longmapsto (q,[-3f^2],[2f^3]). \]
In particular, this shows that the $\GL_3\times\bG_m$-equivariant fundamental class of $\psi(W^N_{2N})$ is equal to the $\PGL_2\times\bG_m$-fundamental class of $\Delta_N^1$.

As in the case where $N$ is odd, we plan to use the localization formula to compute the image of $\ch_*^{G_N}(\Delta_N^1)\to\ch_*^{G_N}(V^N_{4N,6N})$. To pass to integral coefficients however, it is convenient to work in an ambient space $X$ such that $\ch^*([X/\GL_3\times\bG_m])$ is a free $\ch^*(\cB \GL_3\times\bG_m)$-module. Therefore, for our purposes, we set $X=\bP^5$ to be the projectivization of the $\GL_3\times\bG_m$-scheme $\bA^6$. From \cite{DL}*{Definition 3.1} we know that there exists a locally free sheaf $\overline{W}^m_{2dm}$ whose pullback along $\bA^6\smallsetminus\{0\}\to\bP^5$ is isomorphic to $W^m_{2dm}$. Points in the total space of $\overline{W}^m_{2dm}$ are pairs $([q],[f])$, where $[q]=[q']$ if and only if $q=\lambda q'$ for some invertible scalar $\lambda$. We also have an equivariant map $\overline{\psi}:\overline{W}^N_{2N}\rightarrow \overline{W}^N_{4N,6N}$, whose pullback along $\bA^6\smallsetminus\{0\}\to\bP^5$ is isomorphic to $\psi$.
Recall from \cite[4.1]{EdFuRat} that 
\begin{align*}
    \ch_{\GL_3\times\bG_m}^*(\bP^5) &\simeq \oplus_{i=0}^{5} \ch_{\GL_3\times\bG_m}^*(\spec(k))\cdot h^i\\
    &\simeq \bZ[\tau_1,h,c_1,c_2,c_3]/((h^3-2c_1h^2+4c_2h-8c_3)(h^3-2c_1h^2+(c_1^2+c_2)h+c_3-c_1c_2))    
\end{align*}
where $h=c_1^{\GL_3\times\bG_m}(\cO(1))$. Observe in particular that this Chow ring is free as $\ch_{\GL_3\times\bG_m}^*(\spec(k))$-module.

Now we explain how to compute the Chern classes of $\overline{W}^m_{2dm}$. For this, the basic ingredient is the short exact sequence
\[0\to\cO_{\bP^5}(-1)\otimes \Sym^{m-2} E^{\vee} \otimes L^{\otimes(-m)} \to  \Sym^m E^{\vee}\otimes L^{\otimes(-m)}\otimes\cO_{\bP^5} \to \overline{W}^m_{2m}\to 0 \]
of locally free sheaves on $\bP^5$ (see \cite{DL}*{2.3}), where $E$ is the standard $\GL_3$-representation and $L$ is the standard $\bG_m$-representation of weight one. This implies 
\begin{align}\label{eq:class W}
    c_{2m+1}^{\GL_3\times\bG_m}(\overline{W}^m_{2m})=&\left\{ \frac{c^{\GL_3\times\bG_m}(\Sym^m E^{\vee}\otimes L^{\otimes(-m)})}{c^{\GL_3\times\bG_m}(\Sym^{m-2} E^{\vee} \otimes L^{\otimes(-m)}\otimes\cO(-1))} \right\}_{2m+1}\\
    &=\left\{\frac{\prod_{i+j\leq m}(1+x(i\ell_1+j\ell_2+(m-i-j)\ell_3-m\tau_1))}{\prod_{i'+j'\leq m-2}(1+x(i'\ell_1+j'\ell_2+(m-2-i'-j')\ell_3-m\tau_1-h))}\right\}_{2m+1} \nonumber
\end{align}
This expression in brackets should be interpreted as a formal series in $x$, from which we are extracting the coefficient in front of $x^{2m+1}$. Moreover, the symbols $\ell_1$, $\ell_2$ and $\ell_3$ stands for the Chern roots of $E^{\vee}$, so that the elementary symmetric polynomial in $\ell_1$, $\ell_2$ and $\ell_3$ of degree $d$ is equal to $(-1)^dc_d$.
\begin{Remark}\label{rmk:non zero div}
    Observe that $c_{2m+1}^{\GL_3\times\bG_m}(\overline{W}^m_{2m})$ is not a zero divisor. Indeed, if it was a zero divisor, this would imply that there exists a non zero element $\xi$ in the equivariant Chow ring of $\bP^5$ such that 
    \[ \xi \cdot c_{2m+1}^{\GL_3\times\bG_m}(\overline{W}^m_{2m}) \cdot c_{\frac{m(m-1)}{2}}^{\GL_3\times\bG_m}(\cO_{\bP^5}(-1)\otimes \Sym^{m-2} E^{\vee} \otimes L^{\otimes(-m)}) = \xi \cdot c_{\frac{(m+1)(m+2)}{2}}^{\GL_3\times\bG_m}(\Sym^m E^\vee \otimes L^{\otimes(-m)}) = 0, \]
    which would contradict the fact that $\ch_{\GL_3\times\bG_m}^*(\bP^5)$ is a free $\ch_{\GL_3\times\bG_m}^*(\spec(k))$-module.
\end{Remark}
\begin{Prop}\label{prop:class D1 PGL}
For $N$ even, the image of $\ch_*^{G_N}(\Delta_N^1)\to\ch_*^{G_N}(V^N_{4N,6N})$ is generated as an ideal by $$\left.\frac{c_{10N+2}^{\GL_3\times\bG_m}(\overline{W}^N_{4N,6N})}{c_{2N+1}^{\GL_3\times\bG_m}(\overline{W}^N_{2N})}\right|_{h=c_1=0}.$$
The expression above is a polynomial and after evaluation at $h=c_1=0$, it should be viewed as an element in $\bZ[\tau_1,c_2,c_3]/(2c_3)$, the ${G_N}$-equivariant Chow ring of $V^N_{4N,6N}$.
\end{Prop}
\begin{proof}
The fact that $[\Delta_N^1]_{G_N}$ generates as an ideal the image of $\ch_*^{G_N}(\Delta_N^1)\to\ch_*^{G_N}(V^N_{4N,6N})$ has already been proved in \Cref{envelope}. Moreover, the previous discussion shows that $[\Delta_N^1]_{G_N}=[\psi(W^N_{2N})]_{\GL_3\times\bG_m}$.

Let $T\subset\GL_3\times\bG_m$ be the maximal subtorus of pairs formed by diagonal matrices and an invertible scalar. The fixed points for the action of $T$ on $\overline{W}^m_{2dm}$ are of the form $([q],[0])$ where $q$ is a monomial. Observe that the tangent space of $\overline{W}^m_{2dm}$ at $p$ is isomorphic to the direct sum $T\bP^5_{[q]}\oplus \overline{W}^m_{2dm,[q]}$. Moreover, the fundamental class of $([q],[0])$ in the equivariant Chow ring of $\overline{W}^m_{2dm}$ is equal to the product $[[q]]_T\cdot c_{2dm+1}^T(\overline{W}^m_{2dm,[q]})$. The localization formula (\cite{EGloc}*{Theorem 2}) then gives us the equality
\begin{align}\label{eq:frac}
    \overline{\psi}_*(1)&=\overline{\psi}_*\left( \sum_{q=x_ix_j, i\leq j} \frac{[([q],[0])]_T}{c_{2N+6}^T(T\overline{W}^N_{2N+1,[q]})} \right) \nonumber\\
    &=\sum_{q=x_ix_j, i\leq j} \frac{\overline{\psi}_*([([q],[0])]_T)}{c_5^T(T\bP^5_{[q]})c_{2N+1}( \overline{W}^N_{2N,[q]})} \nonumber\\
    &=\sum_{q=x_ix_j, i\leq j} \frac{[[q]]_T\cdot c_{10N+2}^T(\overline{W}^N_{4N,6N,[q]})}{c_5^T(T\bP^5_{[q]})c_{2N+1}( \overline{W}^N_{2N,[q]})}=\frac{c_{10N+2}^T(\overline{W}^N_{4N,6N})}{c_{2N+1}^T(\overline{W}^N_{2N})},
\end{align}
where in the last equality we applied again the localization formula to obtain an expression in the equivariant Chow ring of $\overline{W}^N_{4N,6N}$. Observe that a priori the localization formulas would only give an equality in the ring $\ch_T^*(\bP^5)\otimes (\ch_T^*(\spec(k))^{+})^{-1}$ obtained by inverting the positive degree elements in $\ch_T^*(\spec(k))$; nevertheless, as $\ch_T^*(\bP^5)$ is a free $\ch^*_T(\spec(k))$-module, the natural homomorphism $\ch_T^*(\bP^5)\to \ch_T^*(\bP^5)\otimes (\ch_T^*(\spec(k))^{+})^{-1}$ is injective: this proves that (1) the last term is not just a rational function but a polynomial, and that (2) it coincides with $\overline{\psi}_*(1)$. Observe moreover that as $c_{2N+1}^T(\overline{W}^N_{2N})$ is not a zero divisor (\Cref{rmk:non zero div}), the expression 
\[ \frac{c_{10N+2}^T(\overline{W}^N_{4N,6N})}{c_{2N+1}^T(\overline{W}^N_{2N})} \]
is well defined, in the sense that it coincides with the unique element $\xi$ such that $\xi \cdot c_{2N+1}^T(\overline{W}^N_{2N}) = c_{10N+2}^T(\overline{W}^N_{4N,6N})$.

The $T$-equivariant top Chern classes of $\overline{W}^m_{2dm}$ are equal to the $\GL_3\times\bG_m$-ones. Observe also that the last term in (\ref{eq:frac}) can be regarded as a polynomial in $h$, the hyperplane class of $\bP^5$, so that the element 
\begin{equation}\label{eq:frac 2}
    \left.\frac{c_{10N+2}^{\GL_3\times\bG_m}(\overline{W}^N_{4N,6N})}{c_{2N+1}^{\GL_3\times\bG_m}(\overline{W}^N_{2N})}\right|_{h=c_1}
\end{equation}
is well defined, and it coincides with the pullback of $\overline{\psi}_*(1)$ along the $\bG_m$-torsor $\bA^6\smallsetminus\{0\}\to\bP^5$, which in turn is equal to $\psi_*(1)$. If we further restrict this cycle to $\bA^6\smallsetminus D$ (observe that this operation sends $c_1$ to zero), we get an explicit expression for the $\GL_3\times\bG_m$-fundamental class of $\psi(W^N_{2N})$, which we already observed to be equal to the $\PGL_2\times\bG_m$-equivariant fundamental class of $\Delta^1_N$. 
\end{proof}

\section{Relations coming from $[\Delta^2_N/G_N]$}\label{Delta2}
In this Section we compute the relations in the Chow ring of $\cW^{\min}_N$ coming from the excision of $[\Delta_N^2/G_N]$. We first define an equivariant stratification for $\Delta_N^2$, which we leverage to compute the generators of the ideal of the relations. The final result is summarized in \Cref{prop:relations D2}.

In the last part of the Section, we prove the first main result of the paper (\Cref{thm:main}). In this section we work over a ground field of any characteristic, with the only exception of \Cref{thm:main}.
\subsection{An equivariant stratification of $\Delta^2_N$}\label{sec:equi}
First, we recall the definition of equivariant stratification in general.
\begin{Def}
Let $X$ be a $G$-scheme. An equivariant stratification of $X$ is a finite family $\{Z_{\tau}\}_{\tau\in J}$ of locally closed, pairwise disjoint, and equivariant subschemes of $X$ such that $\bigcup_{\tau\in J} Z_{\tau}=X$ and 
\[
\overline{Z_{\tau}}\smallsetminus Z_{\tau}=\bigcup Z_{\tau'}.
\]
\end{Def}
We can endow $\bP V_k$ with a $G_N$ action as follows:
\begin{itemize}
    \item for $N$ odd, we can regard $V_k$ as a $\GL_2$-representation with the action defined at the beginning of Section \ref{equivint}; we can then use the isomorphism $G_N \to \GL_2$ in order to give to $V_k$ the structure of a $G_N$-representation. This of course induces an action of $G_N$ on $\bP V_k$. Observe that for $k=rN$, the two $G_N$-representations $V_k$ and $V^N_k$ are \emph{not} the same (which also motivates the difference in the notation).
    \item For $N$ even, we can regard $\bP V_k$ as a $\PGL_2\times\bG_m$-scheme using the $\PGL_2$-action defined at the beginning of Section \ref{equivint}, and letting $\bG_m$ act trivially. Therefore, we can endow $\bP V_k$ with the structure of a $G_N$-scheme via the isomorphism $G_N \to \PGL_2\times\bG_m$. If $k$ is even, we can also endow $V_k$ with a $G_N$-action, exactly in the same way.
\end{itemize}
Let $\Sigma_k^{(m+1)}$ denote the $(m+1)$-thickening of the subscheme $\Sigma_k\subset\bP V_k\times\bP^1$ defined in Section \ref{second case}, i.e. the subscheme defined by the ideal sheaf $\cI_{\Sigma_k}^{m+1}\simeq \cO_{\bP V_k}(-m-1)\boxtimes \cO_{\bP^1}(-(m+1)k)$. We then have a short exact sequence of $G_N$-equivariant sheaves
\[ 0\longrightarrow\cO_{\bP V_k}(-m-1)\boxtimes \cO_{\bP^1}(-(m+1)k)\longrightarrow \cO_{\bP V_k\times\bP^1}\longrightarrow i_*\cO_{\Sigma_k^{(m+1)}}\longrightarrow 0. \]
We can twist the sequence above by $\pr_2^*\cO_{\bP^1}(2d)$ and push everything down on $\bP V_k$; if we further assume that $\oH^1(\bP^1,\cO_{\bP^1}(2d-(m+1)k)=0$, by cohomology and base chage we obtain the following short exact sequence of $G_N$-equivariant locally free sheaves on $\bP V_k$:
\begin{equation}\label{eq:seq PP}
    0\longrightarrow\cO_{\bP V_k}(-m-1)\otimes V_{2d-(m+1)k} \longrightarrow V_{2d}\otimes\cO_{\bP V_k}\longrightarrow \cP^m_k(\cO_{\bP^1}(2d)) \longrightarrow 0
\end{equation}
where we define $\cP^m_k(\cO_{\bP^1}(2d))$ as the locally free sheaf $\pr_{1*}(\pr_2^*\cO_{\bP ^1}(2d)|_{\Sigma_k^{(m+1)}})$. This bundle coincides with the bundle of principal parts considered in \cite{CK}.

In particular, if we specialize this short exact sequence to the cases $(d,m)=(2N,3)$, $(3N,5)$, we get short exact sequences
\begin{align*}
    0\longrightarrow\cO_{\bP V_k}(-4)\otimes V_{4(N-k)} \longrightarrow V_{4N}\otimes\cO_{\bP V_k}\longrightarrow \cP^3_k(\cO_{\bP^1}(4N)) \longrightarrow 0 \\
    0\longrightarrow\cO_{\bP V_k}(-6)\otimes V_{6(N-k)} \longrightarrow V_{6N}\otimes\cO_{\bP V_k}\longrightarrow \cP^5_k(\cO_{\bP^1}(6N)) \longrightarrow 0
\end{align*}
Define $\bL$ as follows:
\begin{itemize}
    \item for $N$ odd, it is defined as $\det(E)^{\otimes \frac{N-1}{2}}$, where $E$ is the standard representation of $\GL_2$;
    \item for $N$ even, it is defined as $L^{\otimes(-1)}$, the rank one representation of $\bG_m$ of weight $-1$.
\end{itemize}
There is an action of $\GL_2/\bmu_N$ on $\bP V_k$: for $N$ odd, we have $\GL_2/\bmu_N\simeq\GL_2$, and for $N$ even we have $\GL_2/\bmu_N\simeq \PGL_2\times\bG_m$; the action of these two groups on $\bP V_k$ coincide with the ones mentioned at the beginning of Section \ref{equivint}.

We have then short exact sequences of $\GL_2/\bmu_N$-equivariant locally free sheaves
\begin{align*}
    0\longrightarrow\cO_{\bP V_k}(-4)\otimes V_{4(N-k)} \otimes\bL^{\otimes 4} \longrightarrow V_{4N}\otimes\bL^{\otimes 4}\otimes\cO_{\bP V_k} \longrightarrow \cP^3_k(\cO_{\bP^1}(4N))\otimes\bL^{\otimes 4} \longrightarrow 0 \\
    0\longrightarrow\cO_{\bP V_k}(-6)\otimes V_{6(N-k)} \otimes\bL^{\otimes 6} \longrightarrow V_{6N}\otimes\bL^{\otimes 6}\otimes\cO_{\bP V_k} \longrightarrow \cP^5_k(\cO_{\bP^1}(6N))\otimes\bL^{\otimes 6} \longrightarrow 0 \\
\end{align*}
Let $Z_k$ denote the total space of the locally free sheaf $$\left(\cO_{\bP V_k}(-4)\otimes V_{4(N-k)} \otimes\bL^{\otimes 4}\right)\oplus \left(\cO_{\bP V_k}(-6)\otimes V_{6(N-k)} \otimes\bL^{\otimes 6}\right).$$
Then $Z_k$ is a $G_N$-equivariant vector subbundle of $V^N_{4N,6N}\times \bP V_k$ and we have an equivariant morphism
\begin{equation*}
    p_k:Z_k \longrightarrow V^N_{4N,6N}
\end{equation*}
whose image corresponds to the invariant subscheme of pair of forms $(A,B)$ such that there exists a form $H$ of degree $k$ with $A$ vanishing with order $\geq 4$ along ${H=0}$ and $B$ vanishing with order $\geq 6$ along ${H=0}$.

Moreover, this morphism is one-to-one on the locally closed subscheme of pairs $(A,B)$ which satisfy the previous condition together with the further restraint that there exists no form $H'$ of degree $k+1$ such that $A$ (resp. $B$) vanish with order $4$ (resp. $6$) along ${H'=0}$.

\begin{Lemma}
The image of the pushforward $\ch^{*-9}([\Delta^2_N/G_N])\longrightarrow\ch^*([V_{4N,6N}^N/G_N])$ is equal to the sum of the images of the equivariant pushforwards $p_{k*}$, for $k=1,\ldots,N$.
\end{Lemma}

\begin{proof}
Set $\Delta^2_{N,k}:=\im(p_k)$, so that we have an equivariant stratification of $\Delta^2_N$ given by
\[ \Delta_N^2=\Delta_{N,1}^2\supset\Delta_{N,2}^2\supset\dots\supset\Delta^2_{N,N-1}\supset\Delta^2_{N,N}. \]
Observe that the induced maps 
$Z_k\smallsetminus p_k^{-1}(\Delta^2_{N,{k+1}})\longrightarrow (\Delta^2_{N,k}\smallsetminus\Delta^2_{N,k+1})$ are equivariant Chow envelopes of the strata. We can then apply \cite{DLFV}*{Lemma 3.3} and conclude the proof.
\end{proof}

We have reduced the problem of computing the relations coming from $\Delta^2_N$ to determining the images of several pushforwards. The generators of the Chow groups of $[Z_k/G_N]$ are easier to compute, and so are their pushforwards. Indeed, consider the diagram
\[ \begin{tikzcd}
\left[V^N_{4N,6N}\times\bP V_k/G_N\right] \ar[r, "\pr_1"] \ar[d, "\pr_2"] & \left[V^N_{4N,6N}/G_N\right] \\
\left[\bP V_k/G_N\right].
\end{tikzcd} \]
Then we have the following.
\begin{Lemma}\label{lm:generators}
The image of $p_{k*}$ is generated as an ideal by all the cycles of the form $\pr_{1*}([Z_k]_{G_N}\cdot \pr_2^*\eta)$, where $\eta$ ranges among all the generator of $\ch^*([\bP V_k/G_N])$ as $\ch^*(\cB G_N)$-module.
\end{Lemma}
\begin{proof}
Write $p_k$ as the composition of the closed embedding $i:Z_k\hookrightarrow V^N_{4N,6N}\times\bP V_k$ followed by the projection $\pr_1:V^N_{4N,6N}\times\bP V_k\to V^N_{4N,6N}$. Observe that the Chow ring of $[Z_k/G]$ is generated as a module over $\ch^*(\cB G)$ by the pullback of generators of $\ch^*(\bP V_k)$, i.e. by elements of the form $i^*\pr_2^*\eta$. We deduce that the image of $p_{k*}$ is generated as an ideal by
\begin{align*}
    p_{k*}(i^*\pr_2^*\eta) = \pr_{1*}i_*(i^*\pr_2^*\eta) = \pr_{1*}([Z_k]_G\cdot\pr_2^*\eta),
\end{align*}
as claimed.
\end{proof}

\subsection{Computation of the fundamental class of $Z_k$}.

The subvariety $Z_k\subset V_{4N,6N}^N \times  \bP V_k$ has codimension $10k$ and its equivariant fundamental class is equal to the equivariant top Chern class of the vector bundle 
\begin{equation*}
    \left(\cP^3_k(\cO_{\bP^1}(4N))\otimes\bL^{\otimes 4}\right) \oplus  \left(\cP^5_k(\cO_{\bP^1}(6N))\otimes\bL^{\otimes 6}\right),
\end{equation*} 
which is equal to the product of the top Chern classes of the two factors. We write
\begin{equation*}
    c_{2dk}^{G_N}(\cP^{2d-1}_k(\cO_{\bP^1}(2dN))\otimes\bL^{\otimes 2d})=\sum_{i=0}^{2dk} c_i^{G_N}(\cP^{2d-1}_k(\cO_{\bP^1}(2dN)))(2d\zeta_1)^{2dk-i},
\end{equation*}
where we set $\zeta_1=-\tau_1$ when $N$ is even and $\zeta_1=(N-1)c_1/2$ when $N$ is odd.
In this way, we have reduced our computation of the fundamental class of $Z_k$ to determining the $G_N$-equivariant Chern classes of $\cP^{2d-1}_k(\cO_{\bP^1}(2dN))$. For this we use (\ref{eq:seq PP}), which tells us that
\begin{equation*}
    c_i^{G_N}(\cP^{2d-1}_k(\cO_{\bP^1}(2dN)))=\sum_{j=0}^{i} c_j^{G_N}(V_{2dN})s_{i-j}^{G_N}(V_{2d(N-k)}\otimes\cO_{\bP V_k}(-2d)).
\end{equation*}
Define $\xi_1$ as $c_1^{G_N}(\cO_{\bP V_k}(1))$ for $N$ odd or $N$ even and $k$ even, and as $\frac{1}{2}c_1^{G_N}(\cO_{\bP V_k}(2))\gamma_1/2$ for $N$ even and $k$ odd. Observe that $\cO_{\bP V_k}(2)$ is indeed an equivariant $G_N\simeq\PGL_2\times\bG_m$-line bundle: this follows from the fact that the square map $\bP V_k\hookrightarrow\bP V_{2k}$ is $\PGL_2$-equivariant, and the restriction of the $\PGL_2$-equivariant line bundle $\cO_{\bP V_{2k}}(1)$ (which is equivariant because $\bP V_{2k}$ is the projectivization of the $\PGL_2$-representation $V_{2k}$) along this map coincides with $\cO_{\bP V_k}(2)$. 

Applying the formula for the Segre classes of tensor products, we obtain
\begin{align*}
    s_{i-j}^{{G_N}}(V_{2d(N-k)}\otimes\cO_{\bP V_k}(-2d))&=\sum_{\ell=0}^{i-j}(-1)^{i-j-l}\binom{2d(N-k)+i-j}{2d(N-k)+\ell}s_{\ell}^{{G_N}}(V_{2d(N-k)})(-2d\xi_1)^{i-j-\ell}\\
    &=\sum_{\ell=0}^{i-j}\binom{2d(N-k)+i-j}{2d(N-k)+\ell}s_{\ell}^{{G_N}}(V_{2d(N-k)})(2d\xi_1)^{i-j-\ell}
\end{align*}
Putting everything together, we get the following expression for the ${G_N}$-equivariant fundamental class of $Z_k$:
\begin{align}\label{eq:class Zk}
    \prod_{d=2}^{3}\sum \binom{2d(N-k)+i_d-j_d}{2d(N-k)+\ell_d}(2d)^{2dk-j_d-\ell_d} c_j^{{G_N}}(V_{2dN})s_{\ell_d}^{G_N}(V_{2d(N-k)})\xi_1^{i_d-j_d-\ell_d}\zeta_1^{2dk-i_d}
\end{align}
where the sum index runs over all the triples $(i_d,j_d,\ell_d)$ such that $j_d+\ell_d\leq i_d\leq 2dk$, for $d=2,3$.
\subsection{Relations from $\Delta^2_N$}
We are going to compute generators as an ideal for the image of
\begin{equation}\label{eq:push D2}
    \ch^{*-9}([\Delta^2_N/G_N])\longrightarrow\ch^*([V^N_{4N,6N}/{G_N}]).
\end{equation}
Consider again the diagram
\begin{equation}\label{eq:diag PVk}
\begin{tikzcd}
\left[V^N_{4N,6N}\times\bP V_k/{G_N}\right] \ar[r, "\pr_1"] \ar[d, "\pr_2"] & \left[V^N_{4N,6N}/{G_N}\right] \ar[d] \\
\left[\bP V_k/{G_N}\right] \ar[r, "\pi"] & \cB {G_N}.
\end{tikzcd} 
\end{equation}
From \Cref{lm:generators} we know that the image of (\ref{eq:push D2}) is generated by the cycles $\pr_{1*}([Z_k]_{G_N}\cdot\pr_2^*\eta)$, where $\eta$ ranges among all the generators of $\ch^*(\left[\bP V_k/{G_N}\right])$ as $\ch^*(\cB {G_N})$-module.

Let us rewrite the formula for $[Z_k]_{G_N}$ contained in (\ref{eq:class Zk}) as 
\begin{equation*}
    \sum C_k(i,j,\ell)\xi_1^{i-j-\ell}
\end{equation*}
where $j+\ell\leq i\leq 10k$, and the coefficients are
\begin{align*}
    C_k(i,j,\ell)=\zeta_1^{10k-i} \cdot \left(\sum  \prod_{d=2}^{3} \binom{2d(N-k)+i_d-j_d}{2d(N-k)+\ell_d}(2d)^{2dk-j_d-\ell_d} c_{j_d}^{{G_N}}(V_{2dN})s_{\ell_d}^{{G_N}}(V_{2d(N-k)})\right).
\end{align*}
The sum above is taken over all the triples $(i_d,j_d,\ell_d)$, $d=2,3$, such that
\begin{equation*}
    j_d+\ell_d\leq i_d\leq 2dk,\quad \sum_{d=2}^{3} (i_d,j_d,\ell_d)=(i,j,\ell).
\end{equation*}
Pullbacks along the vertical arrows of the diagram (\ref{eq:diag PVk}) induce isomorphism of Chow rings. Thus, after identifying the Chow rings on the top of the diagram with the respective ones on the bottom, we have
\begin{equation}\label{eq:push not exp}
     \pr_{1*}([Z_k]_{G_N}\cdot \pr_2^*\eta) =  \sum C_k(i,j,\ell) \pi_*(\xi_1^{i-j-\ell}\cdot\eta)
\end{equation}
Note that, in the equality above, we are allowed to apply the projection formula because the coefficients $C_k(i,j,\ell)$ are cycles pulled back from $\cB {G_N}$.

For ${G_N}\simeq\GL_2$ we know from \Cref{prop:chow hilb even} that the Chow ring of $[\bP V_k/\GL_2]$ is generated by powers of the hyperplane class $h$ and we have $\zeta_1=(N-1)c_1/2$ and $\xi_1=h$. Therefore, applying \Cref{prop:chow hilb even} we get the following explicit expression for (\ref{eq:push not exp}) when $\eta=h^m$:
\begin{align}\label{eq:relations D2 for GL2}
    f_{k,m}:=&\sum \left((N-1)c_1/2\right)^{10k-i}s^{\GL_2}_{i-j-\ell-(k-m)}(V_k) \nonumber \\
    &\cdot \left(\sum  \prod_{d=2}^{3} \binom{2d(N-k)+i_d-j_d}{2d(N-k)+\ell_d}(2d)^{2dk-j_d-\ell_d} c_{j_d}^{\GL_2}(V_{2dN})s_{\ell_d}^{\GL_2}(V_{2d(N-k)})\right),
\end{align}
where $(i,j,\ell):=(i_2,j_2,\ell_2)+(i_3,j_3,\ell_3)$ and the sum is taken over all the pairs of triples of positive numbers $\{(i_d,j_d,\ell_d)\}_{d=2,3}$ such that $j_d+\ell_d\leq i_d \leq 2dk$.

For ${G_N}\simeq\PGL_2\times\bG_m$ and $k$ even, we have a similar picture: the only difference is that $\zeta_1= -\tau_1$, hence an explicit expression for (\ref{eq:push not exp}) when $\eta=h^{m'}$ is given by
\begin{align}\label{eq:relations D2 for PGL2 k even}
    g_{k,m'}:=&\sum (-\tau_1)^{10k-i}s^{\PGL_2}_{i-j-\ell-(k-m')}(V_k) \nonumber\\ &\cdot \left(\sum  \prod_{d=2}^{3} \binom{2d(N-k)+i_d-j_d}{2d(N-k)+\ell_d}(2d)^{2dk-j_d-\ell_d} c_{j_d}^{\PGL_2}(V_{2dN})s_{\ell_d}^{\PGL_2}(V_{2d(N-k)})\right).
\end{align}
where again $(i,j,\ell):=(i_2,j_2,\ell_2)+(i_3,j_3,\ell_3)$ and the sum is taken over all the pairs of triples of positive numbers $\{(i_d,j_d,\ell_d)\}_{d=2,3}$ such that $j_d+\ell_d\leq i_d \leq 2dk$.

Finally, for $G_N\simeq\PGL_2\times\bG_m$ and $k$ odd, we know from \Cref{prop:chow hilb odd} that the Chow ring of the stack $[\bP V_k/\PGL_2\times\bG_m]$ is generated as a module over $\ch^*(\cB (\PGL_2\times\bG_m))$ by monomials of the form $\gamma_1^m\gamma_2^n$, where $m\in\{0,1\}$ and $n\leq \frac{k-1}{2}$. Moreover, for $k$ odd we have $\xi_1=\frac{\gamma_1}{2}$, hence
\begin{equation*}
    \pr_{1*}([Z_k]_{G_N}\cdot \gamma_1^m\gamma_2^n) =  \sum 2^{-(i-j-\ell)}C_k(i,j,\ell) \pi_*(\gamma_1^{m+i-j-\ell}\gamma_2^n)
\end{equation*}
Write $m'=2n+m$, where $m$ is either $0$ or $1$. Applying \Cref{lm:push odd}, we get the following explicit expression for the pushforwards:
\begin{align}\label{eq:relations D2 for PGL2 k odd}
    g_{k,m'}:=&\sum k^{-1} 2^{-(i-j-\ell)}(-\tau_1^{10k-i})\\
    &\cdot\left(\sum  \prod_{d=2}^{3} \binom{2d(N-k)+i_d-j_d}{2d(N-k)+\ell_d}(2d)^{2dk-j_d-\ell_d} c_{j_d}^{\PGL_2}(V_{2dN})s_{\ell_d}^{\PGL_2}(V_{2d(N-k)})\right) \nonumber \\
    &\cdot\left(\sum_{q\leq n+\frac{m+i-j-\ell-k}{2}} E_{n,m+i-j-\ell}(q) \cdot s^{\PGL_2}_{2(n-q)+m+i-j-\ell-k}(V_{k-1})2c_2^q\right) \nonumber
\end{align}
where, as before, we set $(i,j,\ell):=(i_2,j_2,\ell_2)+(i_3,j_3,\ell_3)$ and the sum is taken over all the pairs of triples of positive numbers $\{(i_d,j_d,\ell_d)\}_{d=2,3}$ such that $j_d+\ell_d\leq i_d \leq 2dk$. The quantity $E_{n,m}(q)$ is the one defined just before \Cref{lm:push odd}.

Putting all together, we deduce the following.
\begin{Prop}\label{prop:relations D2}
The image of the pushforward $\ch^{*-9}([\Delta_N^2/{G_N}])\to\ch^*([V^N_{4N,6N}/{G_N}])$ is generated by:
\begin{enumerate}
    \item when $N$ is odd, by the cycles $f_{k,m}$ described in (\ref{eq:relations D2 for GL2}) for $1\leq k\leq N$ and $0\leq m\leq k$;
    \item when $N$ is even, by the cycles $g_{k,m'}$ described in (\ref{eq:relations D2 for PGL2 k even}) and (\ref{eq:relations D2 for PGL2 k odd}) for $1\leq k\leq N$ and $0\leq m'\leq k$.
\end{enumerate}

\end{Prop}

\subsection{Proof of the main result}
We have all the ingredients necessary to prove our main result. Indeed, we know from \Cref{thm:presentation WN} that the stack $\cW^{\min}_N$ is isomorphic to $[V_{4N,6N}^N\smallsetminus (\Delta_N^1\cup \Delta_N^2)/{G_N}]$, hence we have a localization exact sequence
\[ \ch([(\Delta_N^1\cup\Delta_N^2)/{G_N}]) \longrightarrow \ch([V^N_{4N,6N}/{G_N}]) \longrightarrow \ch(\cW^{\min}_N) \longrightarrow 0. \]
The image of the map on the left is equal to the sum of the images of the maps $\ch_*([\Delta_N^i/{G_N}]) \to \ch_*([V^N_{4N,6N}/{G_N}])$ for $i=1,2$, which have been computed in \Cref{envelope} and \Cref{prop:relations D2}.

The integral Chow ring of $[V^N_{4N,6N}/{G_N}]$ is isomorphic to the one of $\cB {G_N}$, where the isomorphism is induced by the pullback morphism along the map $[V^N_{4N,6N}/{G_N}]\to\cB {G_N}$. When $N$ is odd, we have ${G_N}\simeq\GL_2$ and $\ch^*(\cB\GL_2)\simeq \bZ[c_1,c_2]$, with $c_1$ and $c_2$ the Chern classes of the universal rank two vector bundle.

Therefore, the generators $c_1$ and $c_2$ of $\ch^*(\cW_N^{\min})$ are by construction the Chern classes of the pullback of the universal rank two vector bundle on $\cB\GL_2$. The map $\cW_N^{\min}\to\cB\GL_2$ is induced by the rank two vector bundle $\cE_N$ of \Cref{def:vec N even} (see \Cref{prop:class map odd}), hence the pullback of the universal vector bundle is equal to $\cE_N$.

Similarly, for $N$ even we have ${G_N}\simeq\PGL_2\times\bG_m$ and the integral Chow ring of the associated classifying stack is isomorphic to $\bZ[\tau_1,c_2,c_3]/(2c_3)$.

The generator $\tau_1$ is the first Chern class of the pullback of the universal line bundle on $\cB\bG_m$, which by \Cref{prop:class map even} is equal to $\cL_N$. The other two generators $c_2$ and $c_3$ are by definition the pullback of the generators of $\ch^*(\cB\PGL_2)$, which are the Chern classes of the rank three vector bundle $(\cP\overset{p}{\to} B)\longmapsto p_*(\omega_{\cP/B}^{\vee})$. The pullback of the latter is by definition the rank three vector bundle $\cE_N$ of \Cref{def:vec N even}.

Putting all together, we obtain our first main result.
\begin{theorem}\label{thm:main}
Suppose that the ground field has characteristic $\neq 2,3$. Then
\begin{enumerate}
    \item for $N$ odd we have
    \[ \ch^*(\cW^{\min}_N)\simeq \bZ[c_1,c_2]/I_N \]
    where the ideal of relations $I_N$ is generated by the polynomials $f_{k,m}$ described in (\ref{eq:relations D2 for GL2}) for $1\leq k\leq N$ and $0\leq m\leq k$, together with the fundamental class $[\Delta_N^1]_{\GL_2}$. The degree of $f_{k,m}$ is $9k+m$ and the degree of $[\Delta_N^1]_{\GL_2}$ is $8N+1$. The generators $c_1$ and $c_2$ are the Chern classes of the rank two vector bundle $\cE_N$ introduced in \Cref{def:vec N odd}.
    
    \item for $N$ even, we have
    \[ \ch^*(\cW^{\min}_N)\simeq \bZ[\tau_1,c_2,c_3]/(2c_3,I_N) \]
    where the ideal of relations $I_N$ is generated by the polynomials $g_{k,m'}$ described in (\ref{eq:relations D2 for PGL2 k even}) and (\ref{eq:relations D2 for PGL2 k odd}) for $1\leq k\leq N$ and $0\leq m'\leq k$, together with the fundamental class $[\Delta_N^1]_{\PGL_2\times\bG_m}$. The degree of $g_{k,m'}$ is $9k+m'$ and the degree of $[\Delta_N^1]_{\PGL_2\times\bG_m}$ is $8N+1$. The generator $\tau_1$ is the first Chern class of the line bundle $\cL_N$ introduced in \Cref{def:vec N even}, and the generators $c_2$ and $c_3$ are the Chern classes of the rank three vector bundle $\cE_N$ introduced in \Cref{def:vec N even}.
    
\end{enumerate}
\end{theorem}
Note the relations appearing in the Theorem above can be made fully explicit: one can apply \Cref{prop:class D1 GL2} and \Cref{prop:class D1 PGL} for computing the fundamental class of $\Delta_N^1$, and \Cref{prop:chern V2m 1} and \Cref{prop:chern V2m 2} to obtain explicit expressions for the Chern and Segre classes of the representations appearing in $f_{k,m}$ and $g_{k,m'}$. Plugging these formulas into the relations, one get the desired description. This is exactly what we will do in the next Section for $N=1,2$.
\section{Integral Chow rings of stacks of rational elliptic surfaces and elliptic K3 surfaces}\label{N12}
In this Section we compute the integral Chow ring of $\cW_1^{\min}$, the moduli stack of rational elliptic surfaces, and of $\cW_2^{\min}$, the moduli stack of elliptic K3 surfaces. The two main results are \Cref{thm:chow W1} and \Cref{thm:chow W2}.
\subsection{The case $N=1$}
A Weierstrass fibration $X\to \bP^1$ with fundamental invariant $N=1$ is a rational surface, obtained by blowing up $\bP^2$ along the base locus of a pencil of cubics. Equivalently, we can think of $X$ as the blow-up of a Del Pezzo surface of degree $1$ along the anticanonical divisor.

The stack $\cW^{\min}_1$ is not Deligne-Mumford because of the presence of objects with infinite dimensional automorphism group \cite[Remark 4.5]{PS}.

\begin{theorem}\label{thm:chow W1}
Suppose that the ground field $k$ has characteristic $\neq 2,3$ and set $r_6=576(30c_1^6+151c_1^4c_2+196c_1^2c_2^2+64c_2^3)$. Then we have
\[ \ch^*(\cW^{\min}_1)\simeq\bZ[c_1,c_2]/(6c_1c_2r_6,c_1^3r_6,c_1^2c_2r_6), \]
where $c_1$ and $c_2$ are the Chern classes of the rank two vector bundle $\cE_1$ introduced in \Cref{def:vec N odd}.
\end{theorem}
\begin{proof}
This is a straightforward application of \Cref{thm:main}. To compute explicitly the Chern classes of the representations involved, one can use \Cref{prop:chern V2m 1}. The Segre classes are then obtained by formally inverting the total Chern classes. Then one can plug in these expressions into the formulas given in (\ref{eq:relations D2 for GL2}) and into the formula given in \Cref{prop:class D1 GL2}.
After performing these computations with Mathematica, we obtain:
\begin{align*}
    [\Delta_1^1]_{\GL_2}&=-3456c_1c_2(30c_1^6+151c_1^4c_2+196c_1^2c_2^2+64c_2^3);\\
    f_{1,0}&=-576c_1^3(30c_1^6+151c_1^4c_2+196c_1^2c_2^2+64c_2^3);\\
    f_{1,1}&=-576c_1^2c_2(30c_1^6+151c_1^4c_2+196c_1^2c_2^2+64c_2^3).
\end{align*}
This concludes the proof.
\end{proof}

\subsection{The case $N=2$}
The stack $\cW_2^{\min}$ can be regarded as the stack of lattice-polarized elliptic K3 surfaces, as explained in the Introduction of \cite{CK}. The coarse space of this moduli stack is particularly interesting and it has been the subject of much work (see for instance  \cite{ MOP, PY}). Here we determine its integral Chow ring.

\begin{theorem}\label{thm:chow W2}
Suppose that the ground field has characteristic $\neq 2,3$. Then we have
\[ \ch^*(\cW^{\min}_2)\simeq\bZ[\tau_1,c_2,c_3]/(2c_3,r_9,r_{10},r_{18},r_{19}) \]
where
\begin{align*}
    r_9=&1152 (691 c_2^4 \tau_1 - 38005 c_2^3 \tau_1^3 + 309568 c_2^2 \tau_1^5 - 
   497520 c_2 \tau_1^7 + 124416 \tau_1^9), \\
   r_{10}=&1152 (30 c_2^5 - 6811 c_2^4 \tau_1^2 + 133495 c_2^3 \tau_1^4 - 
   481528 c_2^2 \tau_1^6 + 327600 c_2 \tau_1^8 - 20736 \tau_1^{10}), \\
   r_{18}=&1152 c_2^5 (108314154642930 c_2^4 + 1045672 c_2^3 \tau_1^2 - 
   89483 c_2^2 \tau_1^4 + 35 c_2 \tau_1^6 - 4 \tau_1^8),\\
   r_{19}=&2304 c_2^6 \tau_1 (118203201 c_2^3 + 180502 c_2^2 \tau_1^2 - 7 c_2 \tau_1^4 + 4 \tau_1^6).
\end{align*}
The generator $\tau_1$ is the first Chern class of the line bundle $\cL_2$ (see \Cref{def:vec N even}), the other generators $c_2$ and $c_3$ are Chern classes of the rank three vector bundle $\cE_2$ (see \Cref{def:vec N even}), whose first Chern class vanishes.
\end{theorem}
We will prove this Theorem by applying \Cref{thm:main} and by explicitly computing the relations in terms of the generators $\tau_1$, $c_2$ and $c_3$.
\begin{Lemma}\label{lemma:Delta12}
\begin{align*} [\Delta_2^1]_{\PGL_2\times\bG_m}=&-995328 \tau_1 (9 c_2^2 + 160 c_2 \tau_1^2 + 256 \tau_1^4) (100 c_2^6 + 5369 c_2^5 \tau_1^2 \\&+ 74074 c_2^4 \tau_1^4 + 400257 c_2^3 \tau_1^6 + 972972 c_2^2 \tau_1^8 + 1061424 c_2 \tau_1^{10} + 419904 \tau_1^{12})  \end{align*}
\end{Lemma}
\begin{proof}
Instead of applying directly the formula of \Cref{prop:class D1 PGL}, we first compute $[\Delta_2^1]_G$ modulo $c_3$, and then we conclude the computation modulo $2$. This trick is inspired by \cite{FV}.

The homomorphism of algebraic groups $\SL_2\times\bG_m\to\PGL_2\times\bG_m$ given on the first factor by quotienting by $\mu_2$ and the second factor by the identity induces a morphism of stacks \[\cB(\SL_2\times\bG_m)\to\cB(\PGL_2\times\bG_m).\] By taking the pullback along this map we get a homomorphism of rings 
\[\ch^*(\cB(\PGL_2\times\bG_m))\simeq\bZ[c_2,c_3,\tau_1]/(2c_3)\to\ch^*(\cB(\SL_2\times\bG_m))\simeq\bZ[c_2,\tau_1],\]
that sends $\tau_1$ to $\tau_1$, the class $c_2$ to $4c_2$ and $c_3$ is sent to zero (see \cite[Proof of Lemma 5.1]{FV}). The pullback of $[\Delta_2^1]_{G_2}$ along this map is equal to $[\Delta_2^1]_{\SL_2\times\bG_m}$, hence if we compute this last class and we substitute $c_2$ with $c_2/4$ we get an expression of $[\Delta_2^1]_{G_2}$ that holds up to multiples of $c_3$.

The same argument of \Cref{prop:class D1 GL2} shows that
\[ [\Delta_2^1]_{\SL_2\times\bG_m}=\frac{c_{22}^{\SL_2\times\bG_m}(V^2_{8,12})}{c_5^{\SL_2\times\bG_m}(V^2_4)}=\frac{c_{9}^{\SL_2\times\bG_m}(V^2_{8})c_{13}^{\SL_2\times\bG_m}(V^2_{12})}{c_5^{\SL_2\times\bG_m}(V^2_4)}. \]
The representation $V^2_{2m}$ is equal to $\Sym^{2m}E^{\vee}\otimes L^{\otimes(-m)}$, where $E$ is the standard $\SL_2$-representation and $L$ is the standard $\bG_m$-representation (of weight one). If $\ell_1$ and $\ell_2$ denote the Chern roots of $E^{\vee}$ and $\tau_1$ is the first Chern class of $L$, we see that the Chern roots of $V^2_{2m}$ are of the form $i\ell_1+(2m-i)\ell_2-m\tau_1$, for $i=0,\ldots,2m$. As the product of the Chern roots is equal to the top Chern class, after some computations and after plugging in the relations $\ell_1+\ell_2=0$ and $\ell_1\ell_2=c_2$, we get
\begin{align}\label{eq:class D1 2} 
[\Delta_2^1]_{\SL_2\times\bG_m}=&-1019215872 (9 c_2^2 + 40 c_2 \tau_1^2 + 16 \tau_1^4) (6400 c_2^6 \tau_1 + 85904 c_2^5 \tau_1^3 \\ &+ 296296 c_2^4 \tau_1^5 + 400257 c_2^3 \tau_1^7 + 
   243243 c_2^2 \tau_1^9 + 66339 c_2 \tau_1^{11} + 6561 \tau_1^{13}). \nonumber
\end{align}
We replace $c_2$ with $\frac{1}{4}c_2$, thus obtaining 
\begin{align}\label{eq:partial}
    -995328 \tau_1 (9 c_2^2 + 160 c_2 \tau_1^2 &+ 256 \tau_1^4) (100 c_2^6 + 5369 c_2^5 \tau_1^2 \nonumber \\&+ 74074 c_2^4 \tau_1^4 + 400257 c_2^3 \tau_1^6 + 972972 c_2^2 \tau_1^8 + 1061424 c_2 \tau_1^{10} + 419904 \tau_1^{12}).
\end{align}
We deduce that $[\Delta_2^1]_{G_2}$ must be equal to the expression in $(\ref{eq:partial})$ plus an element of the form $c_3\eta$, where $\eta$ belongs to $\bZ[\tau_1,c_2,c_3]/(2)$. In particular, the class of $\Delta_2^1$ modulo $2$ is equal to $c_3\eta$.
For computing the class of $\Delta_2^1$ modulo $2$, we first find an element $\xi$ such that
\begin{equation}\label{eq:condition} \xi\cdot c_5^{\GL_3\times\bG_m} (\overline{W}^2_4) = c_{22}^{\GL_3\times\bG_m}(\overline{W}^2_{8,12}). \end{equation}
This task is accomplished by direct computations of the top Chern classes using (\ref{eq:class W}), and then reduction modulo $2$: we find a polynomial $\xi'$ such that $\xi=h\xi'$ satisfies the condition (\ref{eq:condition}), where $h=c_1^{\GL_3\times\bG_m}(\cO(1))$.
We are still not done, because by \cite[page 8]{EdFuRat} the ring $\ch^*([\bP^5/\GL_3\times\bG_m])\otimes\bZ/2$ is isomorphic to $$\bZ[\tau_1 ,c_1, c_2,c_3,h]/(2,h^3(c_1 c_2 + c_3 + c_1^2 h + c_2 h + h^3)),$$
hence the reduction modulo $2$ of 
\begin{equation}\label{eq:fraction 2}\frac{c_{22}^{\GL_3\times\bG_m}(\overline{W}^2_{8,12})}{c_{5}^{\GL_3\times\bG_m}(\overline{W}^2_{4})} \end{equation}
is equal to $\xi$ only up to annihilators of $c_{5}^{\GL_3\times\bG_m}(\overline{W}^2_{4})$. This top Chern class is equal modulo $2$ to $h^2(c_1 c_2 + c_3 + c_1^2 h + c_2 h + h^3)$, so if $\xi''$ is an annihilator of this element, it must be a multiple of $h$ (this can also be checked directly using the tautological exact sequence on $\bP^5$). This shows that the reduction modulo $2$ of (\ref{eq:fraction 2}) is divisible by $h$. As the reduction modulo $2$ of $[\Delta_2^1]_{G_2}$ is equal to (\ref{eq:fraction 2}) evaluated at $h=0$ (see \Cref{prop:class D1 PGL}), we deduce that this reduction is zero, hence $[\Delta_2^1]_{G_2}$ is equal to the expression in (\ref{eq:partial}).
\end{proof}

According to \Cref{thm:main}, we need to compute five other relations. The first two are obtained as follows: let $Z_1\subset V^2_{8,12} \times \bP V_1$ be the subscheme of triples $(A,B,p)$ where $p$ is a point of $\bP^1$ and the form $A$ (resp. the form $B$) vanishes in $p$ with order $\geq 4$ (resp. $\geq 6$). Let $\gamma_1$ be the generator of the $\PGL_2\times\bG_m$-equivariant Chow ring of $\bP V_1$ as a module over $\ch^*(\cB(\PGL_2\times\bG_m))$. Then the first two relations are given by
\[ g_{1,0}:=\pr_{1*}[Z_1]_{\PGL_2\times\bG_m},\quad g_{1,1}:=\pr_{1*}([Z_1]_{\PGL_2\times\bG_m}\cdot \pr_2^*\gamma_1),  \]
where $\pr_1$ (resp. $\pr_2$) is the projection on the first (resp. second) factor.

Formulas for these two relations are given by (\ref{eq:relations D2 for PGL2 k odd}) with $N=2$, $k=1$, $m\in\{0,1\}$ and $n=0$. To make these expressions completely explicit we have to plug in the formulas for Chern classes and Segre classes of $V_4$, $V_6$, $V_8$ and $V_{12}$, which can be extracted from \Cref{prop:chern V2m 2}. After some computations with Mathematica, we get
\begin{align}\label{eq:Z1}
   g_{1,0} &= -1152 (691 c_2^4 \tau_1 - 38005 c_2^3 \tau_1^3 + 309568 c_2^2 \tau_1^5 - 
   497520 c_2 \tau_1^7 + 124416 \tau_1^9) \\
   \label{eq:Z1gamma}
   g_{1,1} &= -1152 (30 c_2^5 - 6811 c_2^4 \tau_1^2 +
   133495 c_2^3 \tau_1^4 - 481528 c_2^2 \tau_1^6 + 327600 c_2 \tau_1^8 - 20736 \tau_1^{10})
\end{align}
Let us recall how the other three relations are obtained: let $Z_2\subset V^4_{8,12}\times\bP V_2$ be the subscheme of triples $(p_1+p_2,A,B)$ such that $p_1+p_2$ is a dimension zero subscheme of $\bP^1$ of length two and $A$ (resp. $B$) vanishes along $p_1+p_2$ with order $\geq 4$ (resp. $6$). If $h$ denotes the hyperplane section of $\bP V_2$ and $\pr_i$ the projection on the $i^{\rm th}$-factor, then the cycles
\[ g_{2,0}:=\pr_{1*}[Z_2]_{\PGL_2\times\bG_m},\quad g_{2,1}:=\pr_{1*}([Z_2]_{\PGL_2\times\bG_m}\cdot \pr_2^*h),\quad g_{2,2}:=\pr_{1*}([Z_2]_{\PGL_2\times\bG_m}\cdot \pr_2^*h^2) \]
are the three relations we are looking for.

Formulas for these relations are given in \Cref{prop:relations D2}: they correspond to the cases $N=2$, $k=2$ and $0\leq m\leq 2$. Observe that in this case the representation $V_{2d(N-k)}$ is trivial, hence the only non-zero Segre class is the one of degree zero, which is equal to one. This means that in the summation we can impose $\ell_d=0$ for $d=2,3$.

To make the formulas completely explicit, we only need to plug in the values of the Chern classes of $V_8$ and $V_{12}$ and of the Segre classes of $V_2$, which are computed as before using \Cref{prop:chern V2m 2}. After some computations with Mathematica, we get
\begin{align*}
    g_{2,0}= &-11943936 (38562300 c_2^9 - 109363770 c_2^8 \tau_1^2+ 134699250 c_2^7 \tau_1^4 - 303690446 c_2^6 \tau_1^6 + 312766535 c_2^5 \tau_1^8\\
    &-259047756 c_2^4 \tau_1^10 + 192326864 c_2^3 \tau_1^{12} -128471616 c_2^2 \tau_1^{14} + 87091200 c_2 \tau_1^{16} - 11943936 \tau_1^{18}),\\
    g_{2,1}=& 23887872 c_2 \tau_1 (37514745 c_2^8 - 64489645 c_2^7 \tau_1^2+ 97095345 c_2^6 \tau_1^4 - 170891502 c_2^5 \tau_1^6 + 142583080 c_2^4 \tau_1^8 \\
    &- 114176800 c_2^3 \tau_1^{10} + 78779520 c_2^2 \tau_1^{12} - 54743040 c_2 \tau_1^{14} + 23887872 \tau_1^{16}),\\
    g_{2,2}=&-c_2\cdot g_{2,0}.
\end{align*}
These five relations, together with the fundamental class $[\Delta_2^1]_{\PGL_2\times\bG_m}$ computed in \Cref{lemma:Delta12}, are all we need to compute the integral Chow ring of $\cW^{\min}_2$.

A quick computation with Mathematica shows that $[\Delta_2^1]_{\PGL_2\times\bG_m}$ belongs to the ideal generated by (\ref{eq:Z1}) and (\ref{eq:Z1gamma}). After further simplifying it via Mathematica, we obtain the presentation given in \Cref{thm:chow W2}.

\bibliographystyle{amsalpha}
\bibliography{bibliography}
\end{document}